\documentclass[a4paper, 10pt, twoside]{article}

\usepackage{amsmath, amscd, amsfonts, amssymb, amsthm, latexsym, url, color, todonotes, bm, framed, rotating, enumerate} 
\usepackage{graphicx}
\usepackage[left=1in, right=1in, top=1.2in, bottom=1in, includefoot, headheight=13.6pt]{geometry}
\usepackage{mathtools}
\input{xy}
\xyoption{all}

\usepackage[colorlinks,breaklinks=true]{hyperref}
\usepackage[figure,table]{hypcap}
\hypersetup{
	bookmarksnumbered,
	pdfstartview={FitH},
	citecolor={black},
	linkcolor={black},
	urlcolor={black},
	pdfpagemode={UseOutlines}
}
\makeatletter
\newcommand\org@hypertarget{}
\let\org@hypertarget\hypertarget
\renewcommand\hypertarget[2]{%
  \Hy@raisedlink{\org@hypertarget{#1}{}}#2%
} 
\makeatother 

\newtheorem{theorem}{Theorem}[section]
\newtheorem{lemma}[theorem]{Lemma}
\newtheorem{corollary}[theorem]{Corollary}
\newtheorem{proposition}[theorem]{Proposition}
\newtheorem*{theorem*}{Theorem}

\theoremstyle{definition}
\newtheorem{definition}[theorem]{Definition}
\newtheorem{remark}[theorem]{Remark}
\newtheorem{example}[theorem]{Example}
\newtheorem{conjecture}[theorem]{Conjecture}
\newtheorem{question}[theorem]{Question}
\newtheorem{key construction}[theorem]{Key construction}
\newtheorem{notat}[theorem]{Notation}

\newcommand{\xysquare}[8]{
\[\xymatrix{
#1 \ar@{#5}[r] \ar@{#6}[d] & #2 \ar@{#7}[d]\\
#3 \ar@{#8}[r] & #4
}\]
}

\newcommand{\bb}{\mathbb}

\newcommand{\comment}[1]{}

\renewcommand{\phi}{\varphi}

\newcommand{\roi}{\mathcal{O}}

\newcommand{\xto}{\xrightarrow}

\renewcommand{\cal}{\mathcal}
\renewcommand{\hat}{\widehat}
\renewcommand{\frak}{\mathfrak}

\renewcommand{\tilde}{\widetilde}

\renewcommand{\ker}{\operatorname{Ker}}
\renewcommand{\projlim}{\varprojlim}

\DeclareMathOperator{\Frac}{Frac}

\DeclareMathOperator{\gr}{gr}

\DeclareMathOperator{\Pic}{Pic}

\DeclareMathOperator{\Spec}{Spec}

\def\et{{\mathrm{\acute{e}t}}}

\newcommand{\CH}{C\!H}

\newcommand{\K}{\hat{\cal K}}

\newcommand{\rcoeq}[3]{\xymatrix{ #1\ar@/^3mm/[r]^f \ar@/_3mm/[r]_g & #2 \ar[l]_e\ar[r] & #3}}

\usepackage[bbgreekl]{mathbbol}
\DeclareSymbolFontAlphabet{\mathbbm}{bbold}

\usepackage{fancyhdr}

\usepackage{sectsty}
\sectionfont{\Large\sc\centering}
\chapterfont{\large\sc\centering}
\chaptertitlefont{\LARGE\centering}
\partfont{\centering}

\begin{document}
\itemsep0pt

\title{Zero-cycles in families of rationally connected varieties}

\author{Morten L\"uders }

\date{}

\maketitle

\begin{abstract} We study zero-cycles in families of rationally connected varieties. We show that for a smooth projective scheme over a henselian discrete valuation ring the restriction of relative zero cycles to the special fiber induces an isomorphism on Chow groups if the special fiber is separably rationally connected. We further extend this result to certain higher Chow groups and develop conjectures in the non-smooth case. Our main results generalise a result of Koll\'ar \cite{Kollar2004}. \end{abstract}
\tableofcontents

\section{Introduction}
In this article we study algebraic cycles in families of rationally connected varieties. Since their introduction in the 90s by Koll\'ar, Miyaoka and Mori (see for example \cite{KollarMiyaokaMori1992} and \cite{Kollar1996}), these varieties have played an important role in algebraic geometry. In order to be mindful of inseparability problems in positive characteristic, we work with the following definition. A smooth and proper variety $X$ over a field $k$ is called separably rationally connected if there is a variety $U$ and a morphism $F:U\times \bb P^1\to X$ such that the induced map
$$F(-,(0:1))\times F(-,(1:0)): U\to X\times X$$
is dominant and separable. Being separably rationally connected implies being  rationally connected and the converse is true if $\mathrm{ch}(k)=0$. If a proper smooth variety $X$ over an algebraically closed field $k$ is separably rationally connected, then any two distinct closed points $p,q\in X$ may be joined by a rational curve, that is are in the image of a morphism $\bb P^1\to X$ \cite[Ch.IV, Thm. 3.9]{Kollar1996}, which implies in particular that $\CH_0(X)\cong \bb Z$. Rationally connected varieties enjoy further remarkable properties which distinguish them from general varieties. One of these features, on which we focus in this article, is that many results about algebraic cycles in families, which usually just hold with finite coefficients, hold integrally for rationally connected varieties. Our main theorem is the following:
\begin{theorem}[Thm. \ref{main_theorem_in_text}]\label{main_theorem} Let $A$ be a henselian discrete valuation ring with residue field $k$. Let $S=\Spec(A)$ and $X$ be a smooth projective scheme over $S$ of relative dimension $d$. Assume that the special fiber $X_k$ of $X$ is separably rationally connected. Let $(X)^{d,g}$ denote the set of codimension $d$ points $x$ of $X$ such that $\overline{\{x\}}$ is flat over $S$. Let $(X)^{d,t}\subset (X)^{d,g}$ denote the subset of points $x$ for which in addition the intersection of $\overline{\{x\}}$ and $X_k$ is transversal.
Then the following statements hold:
\begin{enumerate}
\item The restriction map $\CH^d(X)\to \CH^d(X_k)$ is an isomorphism.
\item Assume that the Gersten conjecture for Milnor K-theory and for higher Chow groups holds for henselian discrete valuation rings. Then for each $j\geq 1$ and $x\in (X_k)^{(d)}$ the image of the map 
$$\ker[\bigoplus_{y\in \Spec(\roi_{X,x})^{(d),g}}K^M_j(y)\xto{\partial} K^M_{j-1}(x)]\to \CH^{d+j}(X_{},j), $$ which arises from the local to global spectral sequence for higher Chow groups, is equal to the image of $\ker[K^M_j(y)\xto{\partial} K^M_{j-1}(x)]$ for some $y\in \Spec(\roi_{X,x})^{(d),t}$.  
\item Let the assumptions be as in (ii) and assume that $X_k$ has a rational point. If we assume further that the degree map $deg\colon\CH^d(X_{k'})\to \bb Z$ is an isomorphism for every finite field extension $k'/k$ (e.g. $k=\bb F_p$ or $\bb C$), then the vertical maps in the commutative diagram 
\begin{equation*}
\xymatrix{
\CH^{d+j}(X_{},j) \ar[d]_\cong  \ar[r]^{}  &  \CH^{d+j}(X_{k},j) \ar[d]^\cong  
\\
\hat K^M_j(A)  \ar[r]^{}   & K^M_j(k) 
}
\end{equation*}
are isomorphisms for all $j\geq 0$. Here $\hat K^M_j$ denotes the improved Milnor K-theory of a ring defined in \cite{Kerz2010}. 
\end{enumerate}
\end{theorem}

The assumptions in $(ii)$ on the Gersten conjecture are always satisfied if $A$ is of equal characteristic. Furthermore, by Proposition \ref{identificationwithZero-cycles with coefficients in Milnor K-theory}(ii) we do not need to assume the Gersten conjecture if $j\leq 1$ or even if $j\leq 2$ if $k$ is finite by Proposition \ref{Gersten_higherChow}.

\begin{corollary}[Cor. \ref{corollary_main_thm_in_text}]\label{corollary_main_thm}
Let the notation be as in Theorem \ref{main_theorem} and denote the generic fiber of $X$ by $X_K$. Then there is a a zigzag of isomorphisms
$$\CH^{d}(X_{k})\xleftarrow{\cong} \CH^{d}(X_{}) \xto{\cong}  \CH^{d}(X_{K}).$$ 
\end{corollary}

Theorem \ref{main_theorem} and Corollary \ref{corollary_main_thm} generalise a theorem of Koll\'ar which was the main motivation for our article: in \cite{Kollar2004}, Koll\'ar proves that, under the same assumptions as in Theorem \ref{main_theorem}(i), the specialisation map
$sp:\CH_0(X_K)\to \CH_0(X_k)$
is an isomorphism if the residue field $k$ is perfect. Another motivation is that, under the same assumptions as in Theorem \ref{main_theorem}, but not assuming that the special fiber $X_k$ is separably rationally connected, there is an isomorphism $\CH^{d+j}(X_{},j,\bb Z/n) \xto{\cong} \CH^{d+j}(X_{k},j,\bb Z/n)$ for $(n,\mathrm{ch}(k))=1$.  For $j=0$ this theorem is due to Saito and Sato \cite[Cor. 9.5]{SS10} if additionally $k$ is finite or algebraically closed. A different proof was given by Bloch in \cite[App.]{EWB16}. In \cite{KerzEsnaultWittenberg2016}, Kerz, Esnault and Wittenberg showed that the assumptions on the residue field can be dropped and extended the result to the semi-stable case. The assumptions on the singularities of the special fiber were relaxed further by Binda and Krishna in \cite{BindaKrishna2021}. Finally, in \cite{Lu17} the author showed the isomorphism for all $j\geq 0$ in the smooth case.

In a similar vein, we show the following theorem:
\begin{theorem}[Thm. \ref{theorem_higher_zero}]\label{main_thm_2_in_text}
Let the notation be as in Theorem \ref{main_theorem}. Then the restriction map
$$res^{d,1}_{\bb Z}:\CH^d(X,1)\to \CH^d(X_k,1)$$
is surjective.
\end{theorem}

Let us give an overview of the structure of the article. In Section \ref{section_Filtrations on higher zero-cycles} we recall the theory of Chow groups with coefficients in Milnor K-theory which is needed for Theorems \ref{main_theorem}(ii) and \ref{main_thm_2_in_text}. For example, if $X_k$ is a variety of dimension $d$ over a field $k$, then for all $j\geq 0$ $$CH^{d+j}(X_k,j)\cong \mathrm{coker}(\oplus_{x\in X_k^{(d-1)}}K_{j+1}^M(x)\rightarrow \oplus_{x\in X_k^{(d)}}K_{j}^M(x)).$$ 
If $k$ is a local field and $X_k$ proper, then the group $CH^{d+1}(X_k,1)$ is often denoted by $SK_1(X_k)$ and appears in the study of higher dimensional local class field theory as the domain of the reciprocity map $\rho_X:SK_1(X_k)\to \pi_1^{\mathrm{ab}}(X_k) $ (see for example \cite{Fo15}).
We also set up a theory of correspondences and their action on Chow groups with coefficients in Milnor K-theory. This allows to prove the following theorem: 
\begin{theorem}[Thm. \ref{theorem_finite_exponent}]\label{theorem_intro_finite_exponent}
If $X_k$ is smooth projective over $k$ and rationally connected, then $$\ker(\CH^{d+j}(X_k,j)\to K^M_j(k))$$ is of finite exponent.
\end{theorem}
This generalises a theorem of Colliot-Th\'el\`ene (for $j=0$) \cite[Prop. 11]{CT2005} to higher zero-cycles and leads us to extend the Bloch-Beilinson conjectures concerning filtrations on Chow groups to higher Chow groups 
(see Section \ref{subsection_BB}).
In Section \ref{section_codim_one} we study the case of relative (semi-stable) curves which is needed in Sections \ref{section_smooth_case} and \ref{section_higher_zero_cycles} in the proofs of our main theorems. The key idea, originally due to Koll\'ar, is to use the deformation theory of rational curves in order to reduce to the case of relative (semi-stable) curves.

It was observed in \cite{KerzEsnaultWittenberg2016} that when studying Chow groups in a family, more precisely on a regular flat and projective scheme $X$ over a henselian discrete valuation ring $A$, the ordinary Chow group of the special fiber $X_k$ should be replaced by a cohomological one. This is related to the fact that there is a restriction map for $K$-groups
$$K_0(\cal X)\to K_0(X_n)$$
induced by the inclusion $X_n\hookrightarrow X$, where $X_n=X\times_AA/\pi^n$ is the thickened special fiber, $\pi$ being a regular parameter. This map is natural in the sense that it comes from the restriction of locally free sheaves. Inspired by this, we conjecture that if the generic fiber $X_K$ of $X$ is
separably rationally connected, then, assuming the Gersten conjecture for Milnor K-theory for the first isomorphism, the restriction map 
$$res^d: CH^{d}(X)\cong H^{d}(X,\hat{\cal K}^M_{d,X}) \to \projlim_n H^{d}(X_k,\hat{\cal K}^M_{d,X_n})$$
is an isomorphism.
In Section \ref{section_d=2} we give some evidence for this conjecture in case that $d=2$:
\begin{theorem}(Thm. \ref{theorem_surfaces})
Let $X$ be as above, $ch(k)=p$ and $ch(K)=0$. Assume that
the reduced special fiber $X_k$ of $X$ is a simple normal crossing divisor.
Assume further that the Gersten conjecture holds for the Milnor K-sheaf $\K^M_{2,X}$ and that the induced map $ \CH^2(X)\cong H^{2}(X,\K^M_{2,X})  \to \projlim_n H^{2}(X_k,\K^M_{2,X_n})$ is surjective.
Then the map
$$res^2: \CH^2(X) \to \projlim_n H^{2}(X_k,\K^M_{2,X_n})$$
is an isomorphism.
\end{theorem}

In Appendix \ref{appendix_Gersten_higher_chow} we recollect some facts about the Gersten conjecture which we need in several places in the article. In Appendix \ref{section_coh_chow} we propose a theory of cohomological higher Chow groups of zero cycles which generalise the Levine-Weibel Chow group \cite{LeWe85}. These groups already appear naturally in Section \ref{section_codim_one} when studying curves. In general we expect them to calculate the lowest filtered piece of Quillen's higher K-groups. 

\begin{notat}
If $A$ is an abelian group, and $n$ a natural number, we denote by $A[n]$ the $n$-torsion of $A$. If $l$ is a prime number, we denote by $A_{l-\mathrm{tors}}$ the $l$-primary torsion of $A$ and by $A_{\mathrm{tors}}$ the entire torsion subgroup of $A$.
\end{notat}

\paragraph*{Acknowledgement.} I would like to thank Christian Dahlhausen, Salvatore Floccari and Stefan Schreieder for helpful comments and discussions and J\'anos Koll\'ar and Kay Rülling for helpful e-mail correspondence. Furthermore, I would like to thank the participants of a seminar on zero-cycles at Hannover, which led to some of the ideas of the article.


\section{Filtrations on higher zero-cycles}\label{section_Filtrations on higher zero-cycles}
\subsection{Chow groups with coefficients: the four basic maps}
We introduce Chow groups with coefficients from \cite{Ro96} and recall some of their properties. Note that in \cite{Ro96} Rost assumes that the base is a field. We need his results for schemes over discrete valuation rings but we formulate them in the largest generality possible. 
In this section let $S$ be an excellent\footnote{In order for the Gersten complex of a noetherian scheme $X$ to be a complex one needs to assume $X$ to be excellent (see \cite{Ka83} for the case of varieties over a field and \cite[Sec. 8.1]{GilleSzamuely2006} for the more general case of excellent noetherian schemes). Throughout the article we assume that the base scheme $S$ is excellent which implies that all the schemes which we consider and which are of finite type over $S$ are excellent.} noetherian scheme and $X$ be a $d$-dimensional scheme of finite type over $S$. Let
$$C_p(X,n):= \bigoplus_{x\in X_{(p)}}K^M_{n+p}(x)$$
and let 
$$d=d_X=\oplus \partial^x_y:C_p(X,n)\xrightarrow{} C_{p-1}(X,n)$$
be the map induced by the tame symbol (see \cite[(3.2)]{Ro96}). It can be shown that these maps fit into a complex
$$C_{d-*}(X,n):= C_d(X,n)\xrightarrow{d} C_{d-1}(X,n)\xrightarrow{d} C_{d-2}(X,n)\to \dots$$
which we call cycle complex with coefficients in Milnor K-theory. In fact this is just the Gersten complex for Milnor K-theory. We consider these complexes homologically with $C_p(X,n)$ in degree $p$ and let $$A_p(X,n):=H_p(C_{d-*}(X,n)).$$ 
By definition of the complex, there are isomorphisms
$$\CH_p(X)\cong A_p(X,-p).$$
\begin{remark}
One could also work with a cohomological indexing, setting
$$C^p(X,n):= \bigoplus_{x\in X_{(p)}}K^M_{n-p}(x)$$
and
$H^{p}(C^{d-*}(X,n))=:A^{p}(X,n)$.
Then
$\CH^p(X)\cong A^p(X,p).$
\end{remark}
But also some higher Chow groups (and therefore motivic cohomology groups) can be identified with Chow groups with coefficients in Milnor K-theory and the latter point of view can sometimes offer advantages in calculations. For the definition of higher Chow groups and their properties we refer to \cite{Bl86} or Section \ref{subsection_BB}. The following proposition tells us in which range such an identification holds.
\begin{proposition}\label{identificationwithZero-cycles with coefficients in Milnor K-theory}
\begin{enumerate}
\item Let $X_k$ be a variety over a field $k$. Then for all $j\geq 0$ $$CH^{d+j}(X_k,j)\cong \mathrm{coker}(\oplus_{x\in X_k^{(d-1)}}K_{j+1}^M(x)\rightarrow \oplus_{x\in X_k^{(d)}}K_{j}^M(x))\cong A_0(X_k,j).$$
Furthermore,
 $$CH^{d}(X_k,1)\cong H_1(\oplus_{x\in X_k^{(d-2)}}K_{2}^M(x)\rightarrow \oplus_{x\in X_k^{(d-1)}}K_{1}^M(x)\rightarrow \oplus_{x\in X_k^{(d)}}K_{0}^M(x))\cong A_1(X_k,0).$$
\item Let $A$ be a henselian discrete valuation ring and $X$ be a regular scheme which is flat over $A$ of relative dimension $d$. Assume that one of the following conditions holds: 
\begin{itemize}
\item[(a)] the Gersten conjecture holds for higher Chow groups of henselian discrete valuation rings, 
or 
\item[(b)] $j\leq 1$.
\end{itemize}
Then for all $j\geq 0$ $$CH^{d+j}(X,j)\cong H_1(\oplus_{x\in X^{(d-1)}}K_{j+1}^M(x)\rightarrow \oplus_{x\in X^{(d)}}K_{j}^M(x)\xrightarrow{\partial_0} \oplus_{x\in X^{(d+1)}}K_{j-1}^M(x) ) \cong A_1(X,j-1).$$
Furthermore, independent of (a) and (b), $A_0(X,j-1)=0$ and
$$A_2(X,-1)\cong CH^d(X,1).$$
\end{enumerate}
\end{proposition}
\begin{proof}
For the first part of (i) see \cite[Prop. 2.3.1]{Lu17}. For the first part of (ii) see \cite[Rem. 2.4.2]{Lu17}. Indeed, if we consider the spectral sequence of \textit{loc. cit.} 
$$\begin{xy} 
  \xymatrix@-1.5em{
  &  & d & d+1 \\
 -j+1 & 0 &      0   &   0      \\
 -j & \dots & \oplus_{x\in X^{(d)}}CH^{j}(k(x),j) \ar[r]^{}  &      \oplus_{x\in X^{(d+1)}}CH^{j-1}(k(x),j-1) \\
 -j-1 & \dots & \oplus_{x\in X^{(d)}}CH^{j}(k(x),j+1) \ar[r]^{}  &   \oplus_{x\in X^{(d+1)}}CH^{j-1}(k(x),j)    \\
  & \dots &      \dots   &   \dots           
  }
\end{xy} $$ 
then the map in line $-j-1$ is always surjective if we assume (a) and the term $CH^{j-1}(k(x),j)=0$ vanishes if we assume (b).
That $A_0(X,j-1)=0$ follows from the surjectivity of the map $\partial_0$.

For the furthermore of (i) and (ii) see \cite[Proof of Cor. 2.7]{Lu16}.
\end{proof}

Our next goal is to define an action of correspondences, i.e. of classical Chow groups of products of schemes, on Chow groups with coefficients in Milnor K-theory. For the definition we need to recall the following four standard (also called basic) maps defined by Rost:
\begin{definition}\label{definition_correspondences}
Let $X$ and $Y$ be schemes of finite type over $S$.
\begin{enumerate}
\item (\cite[(3.4)]{Ro96}) Let $f:X\to Y$ be an $S$-morphism of schemes. Then the pushforward
$$f_*:C_p(X,n)\to C_p(Y,n)$$
is defined as follows:
  \[
    (f_*)^x_y= \left\{
                \begin{array}{ll}
                  N_{k(x)/k(y)} \quad \mathrm{if}\quad k(x)\quad \mathrm{finite} \quad \mathrm{over}\quad k(y),\\
                  0 \quad \quad \quad \quad\quad \mathrm{otherwise}.
                \end{array}
              \right.
  \]
\item (\cite[(3.5)]{Ro96}) Let $g:Y\to X$ be an equidimensional $S$-morphism of relative dimension $s$. Then 
$$g^*:C_p(X,n)\to C_{p+s}(Y,n-s)$$
is defined as follows
  \[
    (g^*)^x_y= \left\{
                \begin{array}{ll}
                  \mathrm{length}(\roi_{Y_x,y})i_* \quad \mathrm{if}\quad \mathrm{codim}_y(Y_x)=0,\\
                  0 \quad \quad \quad \quad\quad \mathrm{otherwise}.
                \end{array}
              \right.
  \]
Here $i:k(x)\to k(y)$ is the inclusion on residue fields and $i_*$ the induced map on Milnor K-theory.
\item (\cite[(3.6)]{Ro96}) Let $a\in \roi^*_X(X)$. Then 
$$\{a\}:C_p(X,n)\to C_p(X,n+1)$$
is defined by
$$\{a\}^x_y(\rho)=\{a\}\cdot \rho$$
for $x=y$ and zero otherwise.

\item (\cite[(3.7)]{Ro96})  Let $Y$ be a closed subscheme of $X$ and $U=X\setminus Y$. Then the boundary map 
$$\partial^U_Y:C_p(U,n)\to C_{p-1}(Y,n)$$
is defined by the tame symbol $\partial^x_y$.
  
\end{enumerate}
\end{definition}

\begin{notat}
We denote by 
$$\alpha:X\Rightarrow Y$$
a homomorphism $\alpha:C_{d-*}(X,n-*)\to C_{d'-*}(X,m-*)$ which is the composition of basic maps defined in Definition \ref{definition_correspondences}. $\alpha:X\Rightarrow Y$ is called a correspondence. 
\end{notat} 
The next lemma shows that a correspondence is compatible with the boundary maps $d_X$ and $d_Y$. The lemma is formulated in \cite[Prop. 4.6]{Ro96} for schemes over fields but the proofs transfer almost word by word to our situation. Indeed, either the the steps of the proof reduce to equal characteristic or to statements which are known for arbitrary discrete valuation rings. For more details we refer to \cite[Sec. 2]{LuedersGerstenKM}.

\begin{lemma}(\cite[Prop. 4.6]{Ro96})\label{lemma_correspondences_compatibility}
Let $X$ and $Y$ be schemes of finite type over $S$.
\begin{enumerate}
\item Let $f:X\to Y$ be proper $S$-morphism. Then $$d_Y\circ f_*=f_*\circ d_X$$
\item Let $g:Y\to X$ be flat $S$-morphism which is equidimensional with constant fiber dimension. Then
$$g^*\circ d_X=d_Y\circ g^*$$
\item Let $a\in\roi_X(X)^{\times}$. Then
$$d_X\circ \{a\}=-\{a\}\circ d_X.$$
\end{enumerate}
\end{lemma} 
We state the following corollary to make it easier to keep track of the indexing.
\begin{corollary}
Let $X$ and $Y$ be schemes of finite type over $S$.
\begin{enumerate}
\item Let $f:X\to Y$ be proper $S$-morphism. Then $f_*$ induces a map on homology $$f_*:A_p(X,n)\to A_p(Y,n)$$
\item Let $g:Y\to X$ be flat $S$-morphism which is equidimensional with constant fiber dimension $s$. Then
Then $g^*$ induces a map on homology $$g^*:A_p(X,n)\to A_{p+s}(Y,n-s).$$
\item Let $a\in\roi_X(X)^{\times}$. Then $\{a\}$ induces a map on homology $$\{a\}:A_p(X,n)\to A_p(X,n+1).$$
\end{enumerate}
\end{corollary}


\subsection{More properties and calculation for $\bb P^1$}
The functorial properties of Rost's Chow groups with coefficients are those of a homology theory.

\begin{proposition}[Localisation](\cite[p.356]{Ro96})\label{localisation_seq}
Let $X$ be a scheme and $Y$ a closed subscheme with complement $U$. Then there is a long exact sequence 
$$\dots\to A_p(Y,n)\to A_p(X,n) \to A_p(U,n) \to A_{p-1}(Y,n)\to..$$
\end{proposition}
\begin{proof}
This follows from the boundary triple $(Y,i,X,j,U)$ having the decomposition $C_p(X,n)=C_p(Y,n)\oplus C_p(U,n).$
\end{proof}

\begin{proposition}[$\bb A^1$-invariance](\cite[Prop. 8.6]{Ro96})\label{A1invariance}
Let $\pi: V\to X$ be an affine bundle of dimension $t$, then the pullback map
$$\pi^*:A_{p}(X,n)\to A_{p+t}(V,n). $$
is an isomorphism.
\end{proposition}

\begin{proposition}[Mayer-Vietoris for closed covers]\label{proposition_MayerVietoris}
Let $X=\bigcup_iX_i$ be a union of pairwise different irreducible varieties. Let $Y=\bigcup_{i\neq j}X_i\cap X_j$. Then there is a long exact sequence 
$$\dots\to A_p(Y,n)\to \bigoplus_i A_p(X_i,n) \to A_p(X,n) \to A_{p-1}(Y,n)\to..$$
\end{proposition}
\begin{proof}
 The statement follows from the short exact sequence of complexes (in which we set $p=\dim X$)
$$\xymatrix{
& 0 \ar[r] \ar[d] &  \bigoplus_i  \bigoplus_{x\in X_{i(p)}}K^M_{n+p}(x) \ar[r] \ar[d] &   \bigoplus_{x\in X_{(p)}}K^M_{n+p}(x) \ar[d]^{} \ar[r] & 0\\
 0 \ar[r] & \bigoplus_{x\in Y_{(p-1)}}K^M_{n+p-1}(x) \ar[r] \ar[d] &  \bigoplus_i  \bigoplus_{x\in X_{i(p-1)}}K^M_{n+p-1}(x) \ar[r] \ar[d] &   \bigoplus_{x\in X_{(p-1)}}K^M_{n+p-1}(x) \ar[d]^{}  \ar[r] & 0 \\
 0 \ar[r] &  \bigoplus_{x\in Y_{(p-2)}}K^M_{n+p-2}(x)  \ar[r] \ar[d] &   \bigoplus_i \bigoplus_{x\in X_{i(p-2)}}K^M_{n+p-2}(x) \ar[r] \ar[d] &   \bigoplus_{x\in X_{(p-2)}}K^M_{n+p-2}(x) \ar[d]^{}  \ar[r] & 0 \\
 &\dots  & \dots & \dots  &  \
}$$
\end{proof}

\begin{proposition}[Descent for blow ups]\label{proposition_descent_blow_ups}
Let $X$ be a smooth $k$-scheme and $Y$ be a smooth closed subscheme. Let $\pi:\tilde{X}\to X$ be the blow up of $X$ along $Y$.
Then there is a long exact sequence 
$$\dots\to A_p(Y,n)\to A_p(\tilde Y,n)\oplus A_p(X,n) \to A_p(\tilde X,n) \to A_{p-1}(Y,n)\to\dots\; .$$
\end{proposition}
\begin{proof}
This follows from the short exact sequence of complexes (with rows indexed by $p$)
$$0\to C_p(Y)\to C_p(\tilde{Y})\oplus C_p(X)\xto{i_*,I()} C_p(\tilde{X})\to 0.$$
\end{proof}

Finally, we use the localisation sequence and $\bb A^1$-invariance to make some calculations for projective space.

\begin{proposition}(\cite[Prop. 8.2.6]{GilleSzamuely2006})\label{proposition_computationP1k}
Let $k$ be a field, then
$$A_0(\bb P^d_k,j) \cong K^M_j(k).$$
\end{proposition}
\begin{proof}
The proposition is clear for $d=0$. For $d=1$ the localisation sequence of Proposition \ref{localisation_seq} gives an exact sequence 
$$0\to A_1(\bb P^1_k,j)\xto{\cong} A_1(\bb A^1_k,j)\to A_0(\Spec(k),j)\to A_0(\bb P^1_k,j)\xto{} A_0(\bb A^1_k,j)\to 0$$
The isomorphism on the left follows from reciprocity (see Remark \ref{reciprocity} and Proposition \ref{proposition_reciprocity}) and $A_0(\bb A^1_k,j)=0$ by Proposition \ref{A1invariance}. This implies the statement for $d=1$.
Next we proceed by induction on $d$. For $d>1$ the localisation sequence gives the exact sequence
$$0= A_1(\bb A^d_k,j)\to A_0(\bb P^{d-1}_k,j)\to A_0(\bb P^d_k,j)\xto{} A_0(\bb A^d_k,j)= 0$$
where $A_0(\bb A^d_k,j)\cong A_{-d}(\Spec(k),j)=0= A_{1-d}(\Spec(k),j)= A_1(\bb A^d_k,j)$ by Proposition \ref{A1invariance}.
\end{proof}

\begin{proposition}\label{proposition_computationP1A}
Let $A$ be a discrete valuation ring, then
$$A_1(\bb P^d_A,j-1) \cong A_1(\Spec(A),j-1)\cong \hat K^M_j(A),$$
assuming the Gersten conjecture for Milnor K-theory for the isomorphism on the right.
\end{proposition}
\begin{proof}
The proposition is clear for $d=0$. 
For $d=1$ the localisation sequence of Proposition \ref{localisation_seq} gives an exact sequence
$$0\to A_2(\bb P^1_A,j-1)\xto{\cong} A_2(\bb A^1_A,j-1)\to A_1(\Spec(A),j-1)\to A_1(\bb P^1_A,j-1)\xto{} A_1(\bb A^1_A,j-1)=0.$$
The isomorphism on the left follows from reciprocity (see Remark \ref{reciprocity} and Proposition \ref{proposition_reciprocity}) applied to the generic fiber of $\bb P^1_A$ and the fact, that the sets of codimension $1$ points of $\bb P^1_A$ and $\bb A^1_A$ contained in the special fiber coincide. The group on the right is zero since $A_1(\bb A^1_A,j-1)\cong A_0(\Spec(A),j)=0$ by Proposition \ref{A1invariance} and \ref{identificationwithZero-cycles with coefficients in Milnor K-theory}(ii). This implies the statement for $d=1$.
Next we proceed by induction on $d$. For $d>1$ the localisation sequence gives the exact sequence
$$0= A_2(\bb A^d_A,j-1)\to A_1(\bb P^{d-1}_A,j-1)\to A_1(\bb P^d_A,j-1)\xto{} A_1(\bb A^d_A,j-1)= 0$$
where $A_1(\bb A^d_A,j)\cong A_{1-d}(\Spec(A),j)=0= A_{2-d}(\Spec(A),j)= A_2(\bb A^d_A,j)$ by Proposition \ref{A1invariance} and \ref{identificationwithZero-cycles with coefficients in Milnor K-theory}(ii).
\end{proof}


\subsection{A higher degree map}\label{section_higher_degree_map}
Let $X_k$ be a proper $k$-scheme with structure map $p:X_k\to \Spec (k)$. Higher Chow groups are covariantly functorial for proper maps. In particular, $p$ induces a push-forward map
$$p_*:\CH^{d+j}(X_k,j)\to \CH^j(k,j)\cong K^M_j(k).$$

If $[F:k]$ is finite, then there is a commutative diagram
$$\xymatrix{
\CH^{d+j}(X_{F},j)  \ar[r] \ar[d] &  K^M_j(F) \ar[d]^{N_{F/k}} \\
\CH^{d+j}(X_{k},j)   \ar[r]^{}  & K^M_j(k),    \
}$$
in which the vertical maps are proper push-forwards. For Milnor K-theory this is the norm map.
Furthermore, if $F$ is an arbitrary field extension of $k$, then the diagram
$$\xymatrix{
\CH^{d+j}(X_{F},j)  \ar[r] &  K^M_j(F) \\
\CH^{d+j}(X_{k},j) \ar[u]  \ar[r]^{}  & K^M_j(k), \ar[u]    \
}$$
in which the vertical maps are flat pullbacks, commutes. 
\begin{definition}
We denote the kernel of $p_*:\CH^{d+j}(X_k,j)\to K^M_j(k)$ by $\CH^{d+j}(X_k,j)_0$. If $j=0$, then we sometimes drop the $j$ from the notation.
\end{definition}

\begin{remark}[Reciprocity]\label{reciprocity}
In view of the isomorphism
$$CH^{d+j}(X_k,j)\cong \mathrm{coker}(\oplus_{x\in X_k^{(d-1)}}K_{j+1}^M(x)\rightarrow \oplus_{x\in X_k^{(d)}}K_{j}^M(x))\cong A_0(X_k,j)$$
from Proposition \ref{identificationwithZero-cycles with coefficients in Milnor K-theory}(i), the existence of the proper push-forward $p_*$ implies Weil reciprocity for curves:
\begin{proposition}(\cite[Prop. 7.4.4]{GilleSzamuely2006})\label{proposition_reciprocity}
Let $C$ be a smooth projective curve over $k$. For a closed point $P$, let $\partial^M_P:K^M_n(k(C))\to K^M_{n-1}(\kappa(P))$ be the tame symbol coming from the valuation on $k(C)$ defined by $P$. Then for all $\alpha\in K^M_n(k(C))$ we have
$$\sum_{P\in C_{(0)}}(N_{\kappa(P)\mid k}\circ \partial^M_P)(\alpha)=0.$$
\end{proposition}
\end{remark}

If $A$ is a discrete valuation ring with quotient field $K$ and residue field $k$ and $X$ a proper $A$-scheme with structural map $p:X\to \Spec (A)$, then $p$ induces a push-forward map
$$p_*:\CH^{d+j}(X,j)\to \CH^j(\Spec (A),j).$$
Assuming the Gersten conjecture for Milnor K-theory and higher Chow groups for a henselian DVR, the group on the right is isomorphic to the $j$-th improved Milnor K-theory $\hat{K}^M_j(A)$ (see Propositions \ref{Gersten_KM} and \ref{Gersten_higherChow}). 

\begin{definition}
We denote the kernel of $p_*:\CH^{d+j}(X,j)\to \CH^j(\Spec (A),j)$ by $\CH^{d+j}(X,j)_0$.
\end{definition}

The diagram 
$$\xymatrix{
\CH^{d+j}(X_{},j)  \ar[r]^{res} \ar[d]^{p_*} & \CH^{d+j}(X_{k},j)   \ar[d]^{p_*} \\
\CH^j(\Spec (A),j) \ar[r]^{sp}  & K^M_j(k),    \
}$$
commutes. Therefore there is a morphism
$$\CH^{d+j}(X,j)_0\to \CH^{d+j}(X_k,j)_0.$$

\begin{proposition}\label{proposition_push_pull}
Let $F/k$ be a finite field extension of degree $m$. Let $X_k$ be a $k$-scheme and $X_F$ its base change to $F$. Then the composition 
$$\CH^{d+j}(X_k,j)\to \CH^{d+j}(X_F,j)\to \CH^{d+j}(X_k,j)$$ 
is multiplication by $m$.
\end{proposition}
\begin{proof}
This is \cite[Cor. 1.4]{Bl86}. Alternatively this follows from Proposition \ref{identificationwithZero-cycles with coefficients in Milnor K-theory}(ii) and the corresponding fact for Milnor K-theory.
\end{proof}

\subsection{Correspondences on Chow groups with coefficients}\label{section_correspondences}
We define an action of classical correspondences on Chow groups with coefficients. The definition is probably well-known to the expert; for example in \cite{Deglise2006}, D\'eglise shows that Chow groups with coefficients are Nisnevich sheaves with transfer. In order to define this action, we need the following proposition, which is one of the main results of Rost in \cite{Ro96}, generalising the pullback map on algebraic cycles defined using the deformation to the normal cone (see \cite{Fulton1998}).

\begin{proposition}
Let $f:Y\to X$ be a morphism of schemes with $X$ smooth over $S$ and $r=\dim_SX-\dim_SY$. Then there is a homomorphism of complexes
$$I(f):C_*(X,n)\to C_{*-r}(Y,n+r)$$
which has all the expected properties of a pullback map of algebraic cycles. It is defined in terms of the four basic maps and therefore induces a map on homology
$$f^*:A_p(X,n)\to A_{p-r}(Y,n+r).$$
\end{proposition}

In \cite[Sec. 14]{Ro96} Rost defines a cross product
$$\times: C_p(Y,r)\times_{\bb Z}C_q(Z,s)\to C_{p+q}(X\times Z,r+s)$$ 
and shows that it induces a product on homology
$$\times: A_p(Y,r)\times_{}A_q(Z,s)\to A_{p+q}(X\times Z,r+s).$$ 

\begin{definition}
Let $\delta: X\to X\times X$ be the diagonal map and $\dim X=d$. We define an intersection product
$$\cup:A_p(X,r)\otimes A_q(X,s)\to A_{p+q-d}(X,r+s+d)$$
by the composition 
$$\cup: A_p(X,r)\times_{}A_q(X,s)\xto{\times} A_{p+q}(X\times X,r+s)\xto{\delta^*} A_{p+q-d}(X,r+s+d) $$
$$x\otimes y\mapsto (\delta_X)^*(x\times y).$$
In particular, if $r=-p$ and $s=-q$, then this coincides with the usual product of algebraic cycles.
\end{definition}
This allows us to define the action of a correspondence.
\begin{definition}[Action of correspondences]\label{definition_correspondence} 
Let $k$ be a field and $X$ and $Y$ be smooth $k$-varieties.
\begin{enumerate}
\item A correspondence between $X$ and $Y$ is a cycle
$$\Gamma\in \CH_q(X\times Y)=A_q(X\times Y,-q).$$
A $0$-correspondence between $X$ and $Y$ is an element of $\CH_{\dim Y}(X\times Y)$.

\item Let $d=\dim X$ and $d'=\dim Y$. Let $X$ furthermore be projective, which implies that the projection $p_2:X\times Y\to Y$ to the second component is proper (we denote the first projection by $p_1$). 
A correspondence $\Gamma$ defines a morphism
$$\Gamma_*:A_p(X,n) \xto{p_1^*} A_{p+d'}(X\times Y,n-d')\xto{\otimes \Gamma} A_{p+d'+q-(d+d')}(X\times Y,n-d'-q+d+d')\xto{{p_2}_*} A_{p+q-d}(Y,n-q+d) $$
$$x\mapsto p_1^*(x)\mapsto p_1^*(x)\otimes \Gamma \mapsto {p_2}_*(p_1^*(x)\otimes \Gamma). $$
Note that for $X=Y$ and $q=\dim Y$, this simplifies to 
$$\Gamma_*:A_p(X,n) \to A_{p}(X,n). $$
\end{enumerate}
\end{definition}

\begin{remark}
More generally, in Definition \ref{definition_correspondence}, we could have defined the action of higher correspondences $\Gamma\in A_p(X\times Y,q)$ for arbitrary $p,q\in \bb Z$.
\end{remark}

\begin{lemma}\label{lemma_moving}
Let $X$ be a regular integral scheme and $Y$ a closed subscheme of $X$. Let $U=X\setminus Y$. Then the map
$$\bigoplus_{x\in U_{(0)}}K^M_j(x)\to A_0(X,j)$$
is surjective.
\end{lemma}
\begin{proof}
Let $y\in Y_{(0)}$ be a closed point and $\alpha\in K^M_j(y)$. We pick a curve $C$, i.e. an integral closed subscheme of dimension $1$, of $X$ with the following properties:
\begin{itemize}
\item[(1)] The generic point $z$ of $C$ is contained in $U$,
\item[(2)] $C$ is regular at $y$.
\end{itemize}
This curve can be constructed as follows: we may assume that $Y$ is a divisor which locally at $y$ is defined by an element $\pi\in \roi_{X,y}$. If $\dim X=d$, then we pick $d-1$ independent regular parameters $x_1,\dots,x_{d-1}\in \roi_{X,y}$ such that $x_i\roi_{X,y}\nsupseteq \pi\roi_{X,y}$ and let $C:=\overline{V(x_1,\dots,x_{d-1})}$ (cf. \cite[Sublem. 7.4]{SS10}). Let $\rho:\tilde{C}\to C$ be the normalisation of $C$ and $\tilde{\alpha}\in K(C)=K(\tilde{C})$ be a lift of $\alpha$. By the Chinese remainder theorem we can pick a rational function $f\in K(C)=K(\tilde{C})$ which is congruent to $1$ at all elements of the set of points $\rho^{-1}(C\cap Y\setminus y)$ and vanishes at the closed point $\rho^{-1}(y)$. Then $\rho_*(\partial_{\tilde{C}}(\{f,\tilde\alpha\}))=\alpha+Z$ with $Z\in \bigoplus_{x\in U_{(0)}}K^M_j(x)$.
\end{proof}

\begin{proposition}
Let $X$ and $Y$ be regular projective and birationally equivalent $k$-varieties. Then 
$$A_0(X,j)\cong A_0(Y,j).$$
In particular, by Proposition \ref{identificationwithZero-cycles with coefficients in Milnor K-theory}(i), $\CH^{d+j}(X,j)\cong \CH^{d+j}(Y,j)$ for $d=\dim X=\dim Y$. 
\end{proposition}
\begin{proof}
We follow the proof of \cite[Prop. 6.3]{CTCoray1979} and \cite[Ex. 16.1.11]{Fulton1998}, where the statement is proved for $j=0$. Let $\Gamma$ be the closure of the graph of the a birational map from $X$ to $Y$. Then $\Gamma$ and its inverse $\Gamma'$ are correspondences from $X$ to $Y$, resp. $Y$ to $X$. We need to show that $\Gamma'_*\circ \Gamma_*=\mathrm{id}_X$ and $\mathrm{id}_Y=\Gamma_*\circ \Gamma'_*$. But $\Gamma'\circ \Gamma$ is the sum of the identity correspondence  and correspondences whose projections to $X$ are of codimension $>0$ in $X$ (and similarly for $Y$). The statement now follows from Lemma \ref{lemma_moving}.
\end{proof}


\subsection{Chow groups with coefficients in Milnor K-theory of rationally connected varieties of degree zero are of finite exponent}
In this section we fix the following notation.
Let $k$ be a field. Let $X_k$ be a proper $k$-scheme. If $F/k$ is a field extension then let $X_F=X_k\times_k F$.
\begin{theorem}\label{proposition_calc_rc_alg_closed}
If $X_k$ is rationally connected, then $$\CH^{d+j}(X_{\bar k},j)\cong K^M_j(\bar k).$$
\end{theorem}
\begin{proof}
By Proposition \ref{identificationwithZero-cycles with coefficients in Milnor K-theory}(i) we have that
$$\CH^{d+j}(X_{\bar k},j)\cong \mathrm{coker}[\bigoplus_{x\in X_{\bar k(1)}}K^M_{j+1}(x)\to \bigoplus_{x\in X_{\bar k(0)}}K^M_{j}(x)].$$
The assertion therefore follows from the definition of rational connectedness and Proposition \ref{proposition_computationP1k}.
\end{proof}
\begin{corollary}\label{corollary_torsion}
If $X_k$ is rationally connected, then $\CH^{d+j}(X_k,j)_0$ is torsion.
\end{corollary}
\begin{proof}
Let $\alpha\in\CH^{d+j}(X_k,j)_0$. It follows from the commutativity of the diagram
$$\xymatrix{
\CH^{d+j}(X_{\bar{k}},j)  \ar[r] &  K^M_j(\bar{k}) \\
\CH^{d+j}(X_{k},j) \ar[u]  \ar[r]^{}  & K^M_j(k), \ar[u]    \
}$$
and Proposition \ref{proposition_calc_rc_alg_closed} that there is a finite field extension $F$ of $k$ such that $\pi^*(\alpha)=0$ for $\pi:X_F\to X_k$. Let $n=[F:k]$. Then, by Proposition \ref{proposition_push_pull}, $n\cdot \alpha=0$.
\end{proof}

The following theorem is due to Colliot-Th\'el\`ene for $j=0$ \cite[Prop. 11]{CT2005}.
\begin{theorem}\label{theorem_finite_exponent}
If $X_k$ is rationally connected, then $\CH^{d+j}(X_k,j)_0$ is of finite exponent.
\end{theorem}
\begin{proof}
Let $\eta$ be the generic point of $X_k$ and $L=K(X_k)$. By a trace argument we may assume that $X_k(k)\neq \emptyset$. Indeed, let $F$ be a finite field extension of $k$ such that $X_F(F)\neq \emptyset$ and assume that $n\cdot \CH^{d+j}(X_F,j)_0=0$ for some $n>0$. Let $m:=[F:k]$. Then the composition
$$\CH^{d+j}(X_k,j)_0\to \CH^{d+j}(X_F,j)_0\to \CH^{d+j}(X_k,j)_0$$ 
is multiplication by $m$ by Proposition \ref{proposition_push_pull} and therefore $nm\cdot \CH^{d+j}(X_k,j)_0=0$. Therefore let $P\in X_k(k)$. Let $\Omega$ be the algebraic closure of $L$. When base changed to $X_\Omega$, the points $\eta$ and $P$ are $R$-equivalent in $X(\Omega)$ and therefore $\eta=P_\Omega\in \CH_0(X_\Omega)$. This equivalence already holds over some finite extension of $L$ and therefore, again by a trace argument, there is an $n>0$ such that $n\cdot(\eta-P_L)=0\in \CH_0(X_L)$. This implies that 
$$n(\Delta-P\times_k X_k)= Z\in \CH^d(X_k\times X_k),$$
where $Z$ is a cycle supported on $X_k\times Y$ and $Y$ is a subscheme of codimension at least $1$ in $X$. By Section \ref{section_correspondences} we get an action of the correspondence
$$n(\Delta-P\times_k X)_*= Z_*:A_0(X_k,j)\to A_0(X_k,j).$$
Lemma \ref{lemma_moving} implies that $Z_*=0$ and therefore $n\Delta_*=n(P\times_k X)_*$. 
But $(P\times_k X)_*(z)=0$ if $z\in \CH^{d+j}(X_k,j)_0$.
\end{proof}

\begin{corollary}\label{corollary_torsion_fe}
Let $K$ be a $p$-adic field and $X_K$ be a smooth projective, geometrically integral variety over $K$. Then the group ${\CH^{d+j}(X_K,j)_0}_{l-\mathrm{tors}}$ is a finite for $l\neq p$. 

\end{corollary}
\begin{proof}
First assume that $X_K$ has a semistable model $X$ whose special fiber $X_k$ is a simple normal crossing divisor. We recall that the \'etale cycle class map
$$\CH^{d+j}(X,j,\bb Z/l^r)\to H_{\et}^{2d+j}(X,\bb Z/l^r(d+j))$$
is an isomorphism for all $j\geq 0$. For $j=0$ this is due to Saito and Sato \cite[Thm. 0.6]{SS10}. For $j=1$ this follows from the Kato conjectures proved by Kerz and Saito \cite[Thm. 9.3]{KeS12}. For $j>1$ this follows from cohomological dimension.
By the localisation sequence for higher Chow groups the sequence
$$\CH^{d+j}(X,j,\bb Z/l^r)\to \CH^{d+j}(X_K,j,\bb Z/l^r)\to\CH^{d+j-1}(X_k,j-1,\bb Z/l^r)$$
is exact. The finiteness of $H_{\et}^{2d+j}(X,\bb Z/l^r(d+j))$ and $\CH^{d+j-1}(X_k,j-1,\bb Z/l^r)$ implies that the group $\CH^{d+j}(X_K,j,\bb Z/l^r)$ is finitely generated. The case of a general $X_K$ now follows from Gabber's refined uniformisation. 

Since by Theorem \ref{theorem_finite_exponent} the group $\CH^{d+j}(X_K,j)_{0l-\mathrm{tors}}$ is of bounded exponent, the map 
$$\CH^{d+j}(X_K,j)_{0l-\mathrm{tors}}\to\CH^{d+j}(X_K,j)/l^r=\CH^{d+j}(X_K,j,\bb Z/l^r)$$
is injective for $r$ sufficiently large. Since, as we have seen above, the group $\CH^{d+j}(X_K,j,\bb Z/l^r)$ is finitely generated, this implies that $\CH^{d+j}(X_K,j)_{0l-\mathrm{tors}}$ is finite.
\end{proof}

\begin{remark} We expect that even more holds, i.e. that the group $\CH^{d+j}(X_K,j)_0$ of Corollary \ref{corollary_torsion_fe} is a direct sum of a finite group and a $p$-primary group of finite exponent. For $j=0$ this is shown in \cite[Cor. 0.5]{SS10}. \end{remark}

\subsection{Bloch-Beilinson conjectures for higher Chow groups}\label{subsection_BB}
Let $X$ be a variety over a field. In \cite{Bl86} Bloch defines a generalisation of Chow groups, so called higher Chow groups, as a candidate for motivic cohomology. These groups are defined as the homology groups of a certain complex
$$\dots\to z^*(X,n+1)\xrightarrow{\partial_{n+1}} z^*(X,n)\xrightarrow{\partial_n} z^*(X,n-1)\xrightarrow{\partial_{n-1}}\dots .$$
Here $\Delta^n:= \Spec(k[t_0,\dots,t_n]/\sum t_i-1)$ and $z^*(X,n)\subset z^*(X\times \Delta^n)$ is the subset of codimension $*$ cycles meeting all faces properly (for more details see \textit{loc. cit.}). More precisely, $$\CH^*(X,n):=\ker[ z^*(X,n)\xrightarrow{\partial_{n}} z^*(X,n-1)]/\text{im}[z^*(X,n+1)\xrightarrow{\partial_{n+1}} z^*(X,n)].$$ 
In particular, $\CH^r(X,0)\cong\CH^r(X)$ for all $r\geq 0$.

Bloch's higher Chow groups have the following properties: let $f:X\to Y$ be a morphism of $k$-varieties.
\begin{enumerate}
\item If $f$ is flat, then there is a pullback map $f^*:\CH^p(Y,q)\to \CH^p(X,q)$.
\item If $f$ is proper, then there is a push forward map $f_*:\CH^{p+d}(X,q)\to \CH^p(Y,q)$, where $d$ is the relative dimension of $X$ over $Y$.
\item There is a product $\CH^p(X,q) \otimes \CH^r(Y,s)\to \CH^{p+r}(X\times Y,q+s)$. If $X=Y$ is smooth, then pulling back along the diagonal gives a product $\CH^p(X,q) \otimes \CH^r(X,s)\to \CH^{p+r}(X,q+s)$.
\end{enumerate}

These properties allow to define an action of correspondences on higher Chow groups:
\begin{definition}
Let $d=\dim X$ and $d'=\dim Y$. Let $X$ furthermore be projective, which implies that the projection $p_2:X\times Y\to Y$ to the second component is proper (we denote the first projection by $p_1$). 
A correspondence $\Gamma\in\CH^a(X\times Y)$ defines a morphism
$$\Gamma_*:\CH^p(X,q) \xto{p_1^*} \CH^{p}(X\times Y,q)\xto{\otimes \Gamma} \CH^{p+a}(X\times Y,q)\xto{{p_2}_*} \CH^{p+a-d}(Y,q) $$
$$x\mapsto p_1^*(x)\mapsto p_1^*(x)\otimes \Gamma \mapsto {p_2}_*(p_1^*(x)\otimes \Gamma). $$
Note that for $X=Y$ and $a=d=\dim X$, this simplifies to 
$$\CH^d(X\times X)\times \CH^p(X,q)\to \CH^p(X,q)$$
$$(\Gamma,x)\mapsto p_{2*}(p_1^*x\cdot\Gamma).$$
\end{definition}

The following conjecture proposes a generalisation of conjectures due to Bloch and Beilinson for Chow groups (see for example \cite[Conj. 2.1]{Jannsen1994} and \cite[Sec. 5]{Beilinson1987}) to higher Chow groups.
\begin{conjecture}\label{conjecture_BB}
For every smooth projective variety $X$ over a field $k$ there exists a filtration
$$\dots\subset F^i\subset\dots\subset F^1\subset F^0=\CH^p(X,q)_{\bb Q}$$
satisfying the following conditions:
\begin{enumerate}
\item[(a)] $F^1\CH^p(X,0)_{\bb Q}=\CH^p(X)_{\rm hom \bb Q}$ for some fixed Weil cohomology $H^\bullet(X)$,
\item[(b)] $F^r\CH^p(X,q)_{\bb Q}\cdot F^s\CH^{p'}(X,q')_{\bb Q}\subset F^{r+s}\CH^{p+p'}(X,q+q')_{\bb Q}$ under the intersection product,
\item[(c)] $F^\bullet$ is respected by $f^*$ and $f_*$ for morphisms $f:X\to Y$,
\item[(d)] $F^i\CH^p(X,q)_{\bb Q}=0$ for $i>>0$.
\end{enumerate}
\end{conjecture}

\begin{remark}
By Corollary \ref{corollary_torsion} we know that if $X_k$ is rationally connected of dimension $d$, then $\CH^{d+j}(X_k,j)_0$ is torsion. Therefore $\CH^{d+j}(X_k,j)_{\bb Q}\cong K^M_n(k)_{\bb Q}$ and $X_k$ satisfies Conjecture \ref{conjecture_BB}(d) for $F^1=F^2=\dots=0$ and $p=d+j,q=j$.
\end{remark}

\subsection{A cohomological theory of cycle complexes with coefficients in Milnor K-theory}\label{section_cohomological_theory_of_cycle_complexes}
In \cite{BO74} Bloch and Ogus define the notion of a Poincar\'e duality theory with supports which consists of a (twisted) homology theory and a (twisted) cohomology theory with supports which coincide for regular schemes by Poincar\'e duality.
We have already remarked that Rost's Chow groups with coefficients in Milnor K-theory have the properties of a Borel-Moore homology theory, i.e, the following:
they are covariantly functorial for proper maps, contravariantly functorial for flat maps, there exists a long exact localisation sequence and they are homotopy invariant.
In order to get a corresponding cohomological theory with supports of Chow groups with coefficients in Milnor K-theory, we suggest to replace the Gersten complex for Milnor K-theory by the Cousin resolution of Milnor K-theory. The cohomology groups of these Cousin complexes will reappear in Appendix \ref{section_coh_chow} where we define cohomological higher Chow groups of zero-cycles.

\begin{definition}
Let $X$ be a noetherian scheme and $\cal F$ a sheaf of abelian groups on $X$. The complex
$$\bigoplus_{x\in X^{(0)}}i_{x*}H^0_x(X,\cal F)\to \bigoplus_{x\in X^{(1)}}i_{x*}H^1_x(X,\cal F)\to \dots\to \bigoplus_{x\in X^{(d)}}i_{x*}H^d_x(X,\cal F) $$
is called the Cousin complex of $\cal F$ and denoted by $C^\bullet_{\cal F}(X)$.

Let $Z$ be a closed subscheme of $X$. The complex
$$\bigoplus_{x\in Z^{(0)}}i_{x*}H^0_x(X,\cal F)\to \bigoplus_{x\in Z^{(1)}}i_{x*}H^1_x(X,\cal F)\to \dots\to \bigoplus_{x\in Z^{(d)}}i_{x*}H^d_x(X,\cal F) $$
is called the Cousin complex of $\cal F$ on $X$ supported on $Z$ and denoted by $C^\bullet_{\cal F}(X)_Z$.
\end{definition}

By \cite[Ch. IV, Prop. 2.3]{Ha66} the Cousin complex depends functorially on $\cal F$ and is therefore contravariantly functorial for morphisms of schemes in the following sense: if $f:X\to Y$ is a morphism of schemes, $\cal F$ a sheaf on $X$ and $\cal G$ a sheaf on $Y$ and if $\cal F\to f_*\cal G$ is a morphism of sheaves, then there is an induced morphism of Cousin complexes
$f^*:C^\bullet_{\cal F}(Y)\to C^\bullet_{\cal G}(X).$ The following proposition is immediate from is the definition.

\begin{proposition}[Localisation]
Let $X$ be a noetherian scheme and $Y$ a closed subscheme of $X$. Let $U=X\setminus Y$ be the open complement. Let $j:U\hookrightarrow X$ denote the inclusion. Then the sequence of complexes
$$0\to C^\bullet_{\cal F}(X)_Z\to C^\bullet_{\cal F}(X)\to j_*C^\bullet_{\cal F}(U)\to 0$$
is exact.
\end{proposition}

\begin{proposition}[Poincar\'e duality]
Let $X$ be a smooth variety over a field $k$. Then
$$A_p(X,n)\cong H^{d-p}(X,\cal K^M_{n+d}).$$
\end{proposition}
\begin{proof}
This follows from the Gersten conjecture and purity for Milnor K-theory.
\end{proof}


\section{Zero cycles in families of relative dimension $1$}\label{section_codim_one}



In this section we study the case of a family of genus zero curves. This will be an essential ingredient of the proof of our main theorem in the next section.
\begin{proposition}\label{proposition_curve_case}
Let $A$ be a henselian discrete valuation ring with local parameter $\pi$ and residue field $k$. Let $X$ be a regular scheme flat and projective over $A$. 
Let $d$ be the relative dimension of $X$ over $A$ and $X_n:=X\times_A{A/\pi^n}$. We also denote the special fiber $X_1$ by $X_k$.
\begin{enumerate}
\item If $X$ is smooth over $A$, $k=\bb C$ and $X_k$ is rationally connected, then the restriction map $$\CH^1(X)\cong \Pic(X)\to  \Pic(X_k)\cong \CH^1(X_k)$$ is an isomorphism.
\item If $d=1$, and $H^1(X_1,\mathcal{O}_{X_1})=0$, then the restriction map  $$\CH^1(X)\cong \Pic(X)\to  \Pic(X_k)\cong \CH^{LW}_0(X_k)$$ is an isomorphism. Here the last group on the right is the Levine-Weibel Chow group (see Definition \ref{definition_Chow_group_LW}).
\end{enumerate}
\end{proposition}
\begin{proof}
We begin with some preliminaries for $(i)$ and $(ii)$. First note that $\Pic(X)\cong \varprojlim \Pic(X_n)$ by \cite[Thm. 5.1.4]{EGA3}. Furthermore, $H^1(X_1,\mathcal{O}_{X_n}^\times)\cong \Pic(X_n)$. 
Let $A_n:=A/\pi^n$ and consider the short exact sequence
$$0\to (\pi^{n-1})\to A_n\to A_{n-1}\to 0.$$
 Since $\roi_{X_n}$ is flat over $A_n$, tensoring with $\otimes_{A_n}\roi_{X_n}$ gives the short exact sequence
$$0\to (\pi^{n-1})\otimes \roi_{X_n}\cong (\pi^{n-1})\roi_{X_n}\cong \roi_{X_1}\to \roi_{X_n}\to \roi_{X_{n-1}}\to 0$$
(the second isomorphism is induced by mapping an element $a\in \roi_{X_{1},x}$ to an element $\tilde{a}\pi^{n-1}$ with $\tilde a\in \roi_{X_n,x}$ a lift of $a$) and in turn the short exact sequence
$$0\to  \roi_{X_1}\to \roi_{X_n}^*\to \roi_{X_{n-1}}^*\to 0.$$
Taking cohomology, we get the exact sequence 
\begin{equation}\label{equation_exact_seq_thickenings}
H^1(X_1,\mathcal{O}_{X_1})\rightarrow H^1(X_1,\mathcal{O}_{X_n}^*)\rightarrow H^1(X_1,\mathcal{O}_{X_{n-1}}^*) \rightarrow H^2(X_1,\mathcal{O}_{X_1}). 
\end{equation}

In the situation of $(i)$, the first and last group are zero, since $X_k$ is rationally connected and $k=\bb C$. This implies the statement.

In the situation of $(ii)$, $H^2(X_1,\mathcal{O}_{X_1})=0$ for dimension reasons and $H^1(X_1,\mathcal{O}_{X_1})=0$ by assumption. This implies that $\Pic(X)\cong \Pic(X_k)$.
Finally, the last isomorphism follows from Lemma \ref{lemma_BK_LW_dim1_pic}.
\end{proof}


\begin{remark}[Relation to blow-ups of closed points]
Proposition \ref{proposition_curve_case}(ii) is closely related to the blow-up formula for the Picard group. Let the notation be as in Proposition \ref{proposition_curve_case}(ii) and $\pi:\tilde{X}\to X$ be the blow up in a closed point $x$. Then the exceptional divisor $E$ is isomorphic to $\bb P^1_{k(x)}$ and $\Pic(\tilde{X})\cong \Pic(X)\oplus \Pic(E)$ and $\Pic(E)\cong \bb Z$. The same holds for the Levine-Weibel Chow group of the special fiber: $\CH^{LW}_0(\tilde X_k)\cong \CH^{LW}_0(X_k)\oplus \bb Z$.
\end{remark}

The main example we have in mind for Proposition \ref{proposition_curve_case}(ii) is that in which the special fiber is a rational comb:

\begin{definition}
A genus zero \textit{comb} over $k$ with $n$ \textit{teeth} is a reduced projective curve of genus zero (i.e. a curve $C$ with $h^1(C,\roi_C)=0$) having $n+1$ irreducible components over $\bar{k}$ and only nodes as singularities. One component, defined over $k$, is called the \textit{handle} (say $C_0$). The other $n$ components, $C_1,\dots,C_n$, are disjoint from each other and intersect $C_0$ in $n$ distinct points. Every $C_i$ is smooth and rational. The curves $C_1,\dots,C_n$ may not be individually defined over $k$.
\end{definition}

\begin{remark}[Residue fields of intersection points of combs]
Being a comb implies that the handles $C_i$ are of the form $\bb P^1_{k(p_i)}$ for some $p_i\in C_0$. This is important when constructing combs by attaching rational curves over nonclosed fields.
Let $C^i=\bigcup_{j=1}^iC_i$.
For all $i$ there is a short exact sequence
$$0\to \roi_{C_{i}}(-p_i)\to \roi_{C^i}\to \roi_{C^{i-1}}\to 0$$
in which $p_i\in C_i$ is the intersection point of $C_i$ and $C^{i-1}$. Taking cohomology, we get the exact sequence
$$H^1(C_i,\roi_{C_{i}}(-p_i))\to H^1(C^i,\roi_{C^{i}})\to H^1(C^{i-1},\roi_{C_{i-1}}).$$
Finally, the group $H^1(C_i,\roi_{C_{i}}(-p_i))$ is zero by standard calculations of the cohomology of projective space since $p_i$ is of degree one on $C_i$.
\end{remark}

\begin{proposition}\label{proposition_higher_curve_case}
Let the notation be as in Proposition \ref{proposition_curve_case}(ii). Assume the Gersten conjecture for Milnor K-theory and for higher Chow groups holds for henselian discrete valuation rings. Let $j\geq 1$.
Let $Z_1,Z_2\subset X$ be two horizontal subschemes which intersect the special fiber $X_k$ transversally in the same component and such that $\Frac(A)\cong K(Z_1)\cong K(Z_2)$. Let $\alpha_{Z_1}\in K^M_j(K(Z_1))$ be supported on $Z_1$ and $\alpha_{Z_2}\in K^M_j(K(Z_2))$ be supported on $Z_2$. Let $p:X\to \Spec(A)$ be the structure map and $p_*$ the proper push-forward defined in Section \ref{section_higher_degree_map}. Assume that $p_*(\alpha_{Z_1}-\alpha_{Z_2})=0$ and $\partial_0(\alpha_{Z_1}-\alpha_{Z_2})=0$. Then $$\alpha_{Z_1}-\alpha_{Z_2}=0\in \CH^{1+j}(X_{},j)\cong H_1(K_{j+1}^M(K(X))\xrightarrow{\partial_1} \oplus_{x\in X^{(1)}}K_{j}^M(x)\xrightarrow{\partial_0} \oplus_{x\in X^{(2)}}K_{j-1}^M(x) ) .$$
\end{proposition}
\begin{proof}
By Proposition \ref{proposition_curve_case} there is an $f\in K(X)$ such that $\mathrm{div}(f)=Z_1-Z_2$. We may further assume that $f$ is congruent to $1$ in the generic points of the special fiber of $X$ (see the proof of Theorem \ref{theorem_higher_zero}). Let $\alpha\in K^M_j(\Frac(A))$ be the element given by $\alpha_{Z_1}$ and $\alpha_{Z_2}$ and $\tilde{\alpha}$ its image in $K^M_j(K(X))$. Then $\partial_1(\{f,\tilde\alpha\})=\alpha_{Z_1}-\alpha_{Z_2}$.
\end{proof}

\begin{remark}
We expect that the restriction map
$$res: \CH^{1+j}(X_{},j) \cong H^1(X,\cal K^M_{j+1})\to  H^1(X_k,\cal K^M_{j+1})\cong \CH^{1+j}(X_{k},j) $$
induces an isomorphism 
$$res: \CH^{1+j}(X_{},j)_0 \cong  \CH^{1+j}(X_{k},j)_0. $$
Indeed, if $X$ is obtained by a finite number of successive blow ups of closed points of $\bb P^1_A$, then this should follow from the Propositions \ref{proposition_computationP1A}, \ref{proposition_MayerVietoris} and \ref{proposition_descent_blow_ups}.

There is no easy way to formulate a generalisation of this expectation to higher (relative) dimension, unless we make the assumptions as in our main theorem \ref{main_theorem}(ii). 
\end{remark}


\section{The smooth case}\label{section_smooth_case}
In this section let $A$ be a henselian discrete valuation ring with function field $K$ and residue field $k$. Let $S=\Spec(A)$ and $X$ be a regular scheme flat and projective over $S$ with generic fiber $X_K$ and special fiber $X_k$. Let $d$ be the relative dimension of $X$ over $S$. The goal of this section is to prove the following theorem:
\begin{theorem}\label{main_theorem_in_text}
Assume that $X$ is smooth over $S$ and that $X_k$ is separably rationally connected. 
\begin{enumerate}
\item The restriction map $\CH^d(X)\to \CH^d(X_k)$ is an isomorphism.
\item Assume $X_k$ has a rational point that the degree map $deg\colon\CH^d(X_{k'})\to \bb Z$ is an isomorphism for every finite field extension $k'/k$ and assume the Gersten conjecture for Milnor K-theory and for higher Chow groups for henselian discrete valuation rings. Then the vertical maps in the commutative diagram 
\begin{equation}\label{diagram_main_thm}
\xymatrix{
\CH^{d+j}(X_{},j) \ar[d]_\cong  \ar[r]^{}  &  \CH^{d+j}(X_{k},j) \ar[d]^\cong  
\\
\hat K^M_j(A)  \ar[r]^{}   & K^M_j(k) 
}
\end{equation}
are isomorphisms for all $j\geq 0$. 
\end{enumerate}
\end{theorem}

We begin by outlining the proof of Theorem \ref{main_theorem_in_text}.
Let $(X)^{d,g}$ denote the set of codimension $d$ points $x$ such that $\overline{\{x\}}$ is flat over $S$, i.e. the cycles in good position. Let $(X)^{d,t}\subset (X)^{d,g}$ denote the subset of points $x$ for which in addition the intersection of $\overline{\{x\}}$ and $X_k$ is transversal. Let $Z^{d,g}(X)=\bigoplus_{x\in (X)^{d,g}}\bb Z$ and $Z^d(X_{k})=\bigoplus_{x\in (X_k)^{(d)}_{}}\bb Z$. Consider the diagram
$$\xymatrix{
\CH^d(X) \ar@/^1pc/[r]^{\tilde{\rho}} &  \CH^{d}(X_k)  \ar[l]_{\tilde{\gamma}}  \\
Z^{d,g}(X) \ar[u]  \ar[r]^{\rho}  & Z^d(X_{k}). \ar[u]  \ar[lu]_\gamma   \
}$$
The idea of the proof of $(i)$ is the following: the two vertical maps exist and are surjective.  The maps $\rho$ and $\tilde{\rho}$, restricting a horizontal one-cycle to the special fiber, also exist and make the outer diagram commute. We show that the diagonal map $\gamma$, induced by lifting closed points to horizontal curves which intersect the special fiber transversally, makes the lower triangle of the above diagram commute. Then we show that $\gamma$ factorises through rational equivalence which gives the map $\tilde{\gamma}$. The map $\tilde{\gamma}$ is surjective and inverse to $\tilde{\rho}$ and therefore an isomorphism. For $j\geq 1$ the situation is slightly more subtle since already the bottom map in Diagram (\ref{diagram_main_thm}) is surjective but not injective. This approach to studying relative zero-cycles was used by Koll\'ar in \cite[Sec. 18]{Kollar2004} assuming, like we do, that $X_k$ is separably rationally connected, and by Kerz, Esnault and Wittenberg in \cite[Sec. 4 and 5]{KerzEsnaultWittenberg2016} working with finite coefficients prime to the residue characteristic but without the assumption of $X_k$ being separably rationally connected.

\begin{proposition}\label{theorem_lifting_well_defined}
Let the notation be as in Theorem \ref{main_theorem_in_text}. Then the diagram
$$\xymatrix{
\CH^d(X)  &   \\
Z^{d,g}(X) \ar[u]  \ar[r]^{\rho}  & Z^d(X_{k})  \ar[lu]_\gamma   \
}$$
commutes.
For each $j\geq 1$ and $x\in (X_k)^{(d)}$ the image of the map 
$$\ker[\bigoplus_{y\in \Spec(\roi_{X,x})^{(d),g}}K^M_j(y)\xto{\partial} K^M_{j-1}(x)]\to \CH^{d+j}(X_{},j) $$ is equal to the image of $\ker[K^M_j(y)\xto{\partial} K^M_{j-1}(x)]$ for some $y\in \Spec(\roi_{X,x})^{(d),t}$. 
\end{proposition}
\begin{proof}
Consider an element 
$$T:=\sum_{i=1}^s r_i[Z_i]\in \rm ker (\rho).$$
We need to show that $T\equiv 0\in \CH^{d}(X)$. In order to do so, we make several reduction steps. 

\textit{Step 1.} We may assume that all the $Z_i$ intersect $X_k$ in the same point $x$ and that $s=2$. Indeed,  let $Z'$ be a lift of $x$ intersecting $X_k$ transversally. Then we may assume that $T$ is of the form
$$[Z]-n[Z'],$$
where $Z$ and $Z'$ are integral and $n$ is the intersection multiplicity of $Z$ with $X_k$, since if $n_i$ is the intersection multiplicity of $Z_i$ with $X_1$, then $T=\sum_{i=1}^s r_i[Z_i]-n_i r_i[Z']$.

\textit{Step 2.} In this step we reduce to the case that $n=1$, i.e. $Z$ also intersects $X_k$ transversally.
Let $\roi_{\tilde Z}$ be the ring of integers in $K(Z)$. Note that $\roi_{\tilde Z}$  is a henselian discrete valuation ring. We base change $X$ along the map $\Spec(\roi_{\tilde Z})\to S$ and denote the base change of $X$ by $X_{\roi_{\tilde Z}}$. Let $Z'_{\roi_{\tilde Z}}\subset X_{\roi_{\tilde Z}}$ be the base change of $Z'$ and $Z_{\roi_{\tilde Z}}\subset X_{\roi_{\tilde Z}}$ the integral closed subscheme corresponding to $\Spec(\roi_{\tilde Z})$. $Z'_{\roi_{\tilde Z}}$ and $Z'_{\roi_{\tilde Z}}$ restrict to the same point in $X_{\roi_{\tilde Z}}$, i.e. their difference is in the kernel of $\rho$. Then, if we have that
$$[Z_{\roi_{\tilde Z}}]-[Z'_{\roi_{\tilde Z}}]\equiv 0\in \CH_1(X_{\roi_{\tilde Z}}),$$
then also $p_*([Z_{\roi_{\tilde Z}}]-[Z'_{\roi_{\tilde Z}}])=[Z]-n[Z']=0 \in \CH_1(X)$, where $p: X_{\roi_{\tilde Z}}\to X$ is the projection.\footnote{The reader may wonder how this step of the argument relates to "Bloch's trick'' in the appendix of \cite{EWB16}. In fact, it is very similar. Bloch takes the normalisation $\tilde{Z}$ of $Z$ and embeds $\tilde{Z}$ and $Z'$ into $\bb P^m_X$ in order to reduce to the case that $Z$ is regular. We could have done the same thing and then base changed along $\tilde{Z}$ in order to further reduce to the case of $n=1$.}

\textit{Step 3.} 
Let $f:\bb P^1_k\to X_k$ be the rational curve with image the point $x$. 
Taking the diagonal map
$$F:\bb P^1_k\to X\times \bb P^1_k,\quad p\mapsto (f(p),p),$$
we may now show that $\tilde Z_{\roi_{Z}}=(Z_{\roi_{Z}},(1:0))$ and $\tilde Z'_{\roi_{Z}}=(Z'_{\roi_{Z}},(0:1))$ are equal in $\CH_1(X\times \bb P^1_{\roi_Z})$. 
Let $C$ denote the curve which is the image of $F$.
Let $x_1= \tilde Z_{\roi_{Z}}\cap X\times \bb P^1_{k}$ and $x_2= \tilde Z'_{\roi_{Z}}\cap X\times \bb P^1_{k}$. By \cite[Thm. 15]{Kollar2004} (up to passing to  a product with $\bb P^m$ for some $m\in \bb N$), there is a genus zero comb over $C^*\subset X\times \bb P^1_k$ over $k$, 
with handle $C$, which is smooth at $z_1$ and $z_2$ and such that $H^1(C^*,\roi_{C^*}(-x_1-x_2)\otimes N_{C^*})=0$. 
By \cite[Thm 51.5]{AraujoKollar2003}, there is a smooth $S$-scheme $U$ and a family of stable genus zero curves containing $\tilde Z_{\roi_{Z}}\cup \tilde Z'_{\roi_{Z}}$ and a point $u\in U(k)$ corresponding to $C^*$ (in the notation of \textit{loc. cit.}, $P=\tilde Z_{\roi_{Z}}\cup \tilde Z'_{\roi_{Z}}$). Indeed, since $X_k$ is smooth and $C^*$ a local complete intersection, there is an exact sequence of sheaves on $C^*$
$$0\to  T_{C^*}(-x_1-x_2) \to T_{X_k}\mid_{C^*}(-x_1-x_2) \to   N_{C^*}(-x_1-x_2) \to 0.$$
This gives an exact sequence 
$$\to H^1(C^*,T_{C^*}(-x_1-x_2)) \to H^1(T_{X_k}\mid_{C^*}(-x_1-x_2)) \to   H^1(N_{C^*}(-x_1-x_2)) \to 0.$$
Using the Mayer-Vietoris sequence for closed covers, one can show that $H^1(C^*,T_{C^*}(-x_1-x_2))=0$ and therefore $H^1(T_{X_k}\mid_{C^*}(-x_1-x_2))=0$.\footnote{Instead of working with the moduli of stable curves of genus zero, we could also use the Hilbert scheme as in \cite{Kollar2004}. In order to achieve that the resulting relative curve contains $\tilde Z_{\roi_{Z}}\cup \tilde Z'_{\roi_{Z}}$, one has to blow up the scheme in these two sections first though.}
Let $s$ be the closed point of $S$. By Hensel's lemma there is a section $\sigma:S\to U$ such that $\sigma(s)=u$. $\sigma(S)$ corresponds to a relative curve $C_S$, which contains $P$, with special fiber $C^*$. The proposition now follows from Proposition \ref{proposition_curve_case} applied to the irreductible component $C_S'$ of $C_S$ which contains $P$. Indeed, 
since $C_S'$ is smooth around $x_1$ and $x_2$, $[\tilde Z_{\roi_{Z}}]-[\tilde Z'_{\roi_{Z}}]$ restricts to zero in $\Pic((C_S')_k)$ and therefore, by Proposition \ref{proposition_curve_case}, also $[\tilde Z_{\roi_{Z}}]-[\tilde Z'_{\roi_{Z}}]=0\in \CH_1({C}_S')$. 

The case $j\geq 1$ is proved similarly, the important part of the case $d=1$ being treated in Proposition \ref{proposition_higher_curve_case}. For the reduction to the case $d=1$ we proceed as above with the following modifications. 
In \textit{Step 1} we start with an additional $\alpha\in K^M_j(k(Z))$ supported on $Z$. We embed $\tilde Z_{\roi_{Z}}$ and $\tilde Z'_{\roi_{Z}}$ as $\tilde Z_{\roi_{Z}}=(Z_{\roi_{Z}},(1:0),(1:0))$ and $\tilde Z'_{\roi_{Z}}=(Z'_{\roi_{Z}},(1:0),(1:0))$ into $X\times \bb P^2_{\roi_Z}$. In particular they intersect the special fiber in the same point, say $p$. As in \textit{Step 3} of the proof of Theorem \ref{theorem_lifting_well_defined} we can find relative curves of genus zero $C_S$ and $C_S'$ containing $\tilde Z_{\roi_{Z}}$ and $\tilde Z'_{\roi_{Z}}$ respectively with the property that $C_S$ and $C_S'$ intersect transversally in $p$. Let $\tilde Z_{\roi_{Z}}''$ the intersection of the two relative curves at $p$. Then we apply Proposition \ref{proposition_higher_curve_case} to $\tilde Z_{\roi_{Z}}-\tilde Z_{\roi_{Z}}''$ and $\tilde Z'_{\roi_{Z}}-\tilde Z_{\roi_{Z}}''$ in $C_S$ and $C_S'$ respectively to move $\alpha$ to $K^M_j(k(Z'))$ supported on $Z'$. 
\end{proof}

\begin{proof}[Proof of Theorem \ref{main_theorem_in_text}]
In the case of (i) it remains to prove that $\gamma$ factorises through $\tilde\gamma$.  Let $\alpha_0\in C_1(X_k,0)=\bigoplus_{x\in X_k^{(d-1)}}K_{1}^Mk(x)$ be supported on some $x\in X_k^{(d-1)}$. As in the proof of \cite[Lem. 7.2]{SS10}, we can find a relative curve $Z\subset X$ (that is of dimension $2$ and flat over $S$) containing $x$ which is regular at $x$  and such that $Z\cap X_k$ contains $\overline{\{x\}}$ with multiplicity $1$. Let $Z_0$ be the special fiber of $Z$ and denote by $\cup_{i\in I} Z_0^{(i)}\cup \overline{\{x\}}$ the union of the pairwise different irreducible components of $Z_0$. Here the irreducible components different from $\overline{\{x\}}$ are indexed by $I$. Let $z$ be the generic point of $Z$. Now as in the proof of \cite[Lem. 2.1]{Lu16}, we can find a lift $\alpha\in k(z)^{\times}$ of $\alpha_0$ which specialises to $\alpha_0$ in $k(x)^{\times}$ and to $1$ in $K(Z_0^{(i)})^{\times}$ for all $i\in I$. Then $\gamma(\partial(\alpha_0))=\partial(\alpha)=0$ in $A_1(X,-1)$ which implies the above factorisation.

In the case of (ii) we can use the assumption on the special fiber to show that $\CH^{d+j}(X,j)$ is generated by Milnor K-groups supported on just one transversal curve lifting a rational point $q$. 
Let $Z_1,Z_2\in (X)^{d,t}$, $p=Z_1\cap X_k$ and $q=Z_2\cap X_k$ and $\alpha_1\in \ker(K^M_{j}(K(Z_1))\xto{\partial}K^M_{j-1}(k(p)))$.
By Proposition \ref{theorem_lifting_well_defined} it is sufficient to show that there is an $\alpha_2\in K^M_{j}(K(Z_2))$ such that $\alpha_1\equiv \alpha_2\in \CH^{d+j}(X,j)$.  
By making a base change along the morphism $Z_1\to S$ we may assume that $K(Z_1)\cong K$. Since $\CH^d(X_{k(p)})\cong\CH^d(X_k)\cong \bb Z$, there is a finite set of pairs of curves and rational functions $\{(C_i\subset X_k,f_i\in K(C_i)^\times)\}$ with the property that $\sum_i div(f_i)=p-q$. We lift each $C_i$ and $f_i$ exactly as in (i) and denote these lifts by $\tilde C_i$ and $\tilde f_i$. Then, by Proposition \ref{theorem_lifting_well_defined}, $\sum_i \partial_{\tilde C_i}(\{\tilde f_i,\alpha_1\})$ moves $\alpha_1\in K^M_{j}(K(Z_1))$ to an element in $\ker(K^M_{j}(K(Z_2))\xto{\partial}K^M_{j-1}(k(q)))$ since $\alpha_1$ is a symbol consisting of global units. Note that by a similar argument $\CH^{d+j}(X_k,j)\cong K^M_j(k)$.
\end{proof}




\begin{corollary}\label{corollary_main_thm_in_text}
Let the notation as in Theorem \ref{main_theorem_in_text}. \begin{enumerate}
\item There is a zigzag of isomorphisms
$$\CH^{d}(X_{k})\xleftarrow{\cong} \CH^{d}(X_{}) \xto{\cong}  \CH^{d}(X_{K}).$$
\item Assume that $\CH^d(X_k)\cong \bb Z$. Then for all $j\geq 0$ the pullback map induces an isomorphism 
$$g^*:\CH^{d+j}(X_{},j) \cong \ker[\partial: \CH^{d+j}(X_{K},j)\to  \CH^{d+j-1}(X_{k},j-1)]. $$
\end{enumerate}
\end{corollary}
\begin{proof}
$(i)$  The fact that the composition
$$\CH^{d-1}(X_{k}) \to \CH^{d}(X_{}) \xto{\cong} \CH^{d}(X_{k})$$
is zero and that the second map is an isomorphism by Theorem \ref{main_theorem_in_text}(i) implies that the first map in the localisation sequence 
$$\CH^{d-1}(X_{k})\to \CH^{d}(X_{}) \xto{g^*}   \CH^{d}(X_{K})\to 0$$
is zero. Therefore $g^*$ is an isomorphism.

$(ii)$ There is a commutative diagram with exact rows
$$\xymatrix{
\CH^{d-1+j}(X_{k},j)\ar[r]^{} & \CH^{d+j}(X_{},j) \ar[d]^\cong   \ar[r]^{g^*}  &  \CH^{d+j}(X_{K},j) \ar[d] \ar[r]^-{\partial} &  \CH^{d+j-1}(X_{k},j-1)  \ar[d]^{\cong} \ar[r] & 0
\\
0\ar[r]^{} & \hat K^M_j(A)  \ar[r]^{}    & K^M_j(K)   \ar[r]^-{\partial} &  K^M_{j-1}(k)  \ar[r]^{} & 0. \
}$$
The first row is the localisation exact sequence for higher Chow groups and the second row is exact by assumption of the Gersten conjecture. The first and third vertical maps are isomorphisms by Theorem \ref{main_theorem_in_text}(ii). By a diagram chase this implies that the map $g^*$ induces the given isomorphism.

\end{proof}


\section{Higher zero-cycles}\label{section_higher_zero_cycles}

Let $A$ be a henselian discrete valuation ring with quotient field $K$ and residue field $k$. Let $X$ be a smooth and projective scheme over $A$ with generic fiber $X_K$ and special fiber $X_k$. Let $d$ be the relative dimension of $X$ over $\Spec(A)$. Generalising Theorem \ref{main_theorem_in_text}(ii) in a different direction, one may ask the following question: 
\begin{question}
Assume that the special fiber $X_k$ of $X$ is separably rationally connected. Is the restriction map
$$res^{d,i}_{\bb Z}:\CH^d(X,i)\to \CH^d(X_k,i)$$
surjective for all $i\geq 0$.
\end{question}
In this section we study the case $i=1$, in which we can make use of the isomorphisms
$$A_2(X,-1)\cong CH^d(X,1) \quad \mathrm{and} \quad A_1(X_k,0)\cong CH^d(X_k,1)$$
of Proposition \ref{identificationwithZero-cycles with coefficients in Milnor K-theory}. Note further that if the Gersten conjecture holds for the Milnor K-sheaf $\cal K^M_{d,X}$, then these groups are isomorphic to $H^{d-1}(X,\cal K^M_{d,X})$ and $H^{d-1}(X_k,\cal K^M_{d,X_k})$ respectively. We recall how the restriction map is defined on Rost's Chow groups with coefficients in Milnor K-theory. 
Let $\pi$ be some fixed local parameter of $A$. We define the restriction map
$$res_\pi: C_p(X,n)\rightarrow C_{p-1}(X_k,n+1)$$
to be the composition
$$res_{\pi}: C_p(X,n)\rightarrow C_{p-1}(X_K,n+1)\xrightarrow{\cdot\{-\pi\}} C_{p-1}(X_K,n+2)\xrightarrow{\partial}C_{p-1}(X_k,n+1).$$
The map $res_\pi$ depends on the choice of $\pi$ but the induced map on homology
$$res: A_p(X,n)\rightarrow A_{p-1}(X_k,n+1)$$
is independent of the choice (for more details see \cite[Sec. 2]{Lu16}). In the following we assume that we have fixed $\pi$ and drop it from the notation.
\begin{theorem}\label{theorem_higher_zero}
Assume that the special fiber $X_k$ of $X$ is separably rationally connected. Then the map
$$res^{d,1}_{\bb Z}:\CH^d(X,1)\to \CH^d(X_k,1)$$
is surjective.
\end{theorem}
\begin{proof}
We closely follow the proof of \cite[Prop. 2.2]{Lu16}. The diagram to consider is the following: 
$$\begin{xy} 
  \xymatrix{
     C_3(X,-1)=\bigoplus_{x\in {X}_{(3)}}K^M_2(x)  \ar[r]^-{res} \ar[d]_{\partial} &  C_2(X_k,0)=\bigoplus_{x\in {X_k}_{(1)}}K^M_2(x) \ar[d]^{\partial}  \\
     C_2(X,-1)=\bigoplus_{x\in {X}_{(2)}}k(x)^* \ar[r]^-{res} \ar[d]_{\partial} &  C_1(X_k,0)=\bigoplus_{x\in {X_k}_{(1)}}k(x)^* \ar[d]^{\partial}  \\
      Z_1(X) \ar[r]^{res}     & Z_0(X_k)  
  }
\end{xy} $$
Here $Z_i(X):=C_{i}(X,-i)$ are just the cycles of dimension $i$. The restriction map in the lowest degree $res:Z_1(X) \to Z_0(X_k)$ agrees with the restriction map on cycles defined by Fulton in \cite[Rem. 2.3]{Fulton1998}.  
We want to show that the middle horizontal map induces a surjection on homology. First note that the map $res:C_2(X,-1)\to C_1(X_k,0)$ is surjective by \cite[Lem. 2.1]{Lu16}. Let $x\in \ker[res:Z_1(X) \to Z_0(X_k)]$. We show that there is a $\xi\in \ker[res:C_2(X,-1)\to C_1(X_k,0)]$ with $\partial(\xi)=x$. By a diagram chase this then implies the theorem. In particular we do not need the upper line of the above diagram.
 
We first treat the case of relative dimension $d=1$ assuming that $X$ is as in Proposition \ref{proposition_curve_case}(ii). In this case the above diagram becomes the following the diagram
$$\begin{xy} 
  \xymatrix{
     C_2(X,-1)=K(X)^*  \ar[r]^-{res} \ar[d]_{\partial} &  C_1(X_k,0)=\bigoplus_{\mu\in X_k^{(0)}}k(\mu)^* \ar[d]^{\partial}  \\
      Z_1(X) \ar[r]^{res}     & Z_0(X_k)  
  }
\end{xy} $$
We consider the following short exact sequence of sheaves:
\begin{equation}\label{eq1}
0\rightarrow \mathcal{O}^*_{X;X_k}\rightarrow \mathcal{M}^*_{X;X_k}\rightarrow Div(X,X_k)\rightarrow 0,
\end{equation}
where $\mathcal{M}^*_{X;X_k} (\text{resp. }\mathcal{O}^*_{X;X_k})$ denotes the sheaf of invertible meromorphic functions (resp. invertible regular functions) relative to $\text{Spec}(A)$ and congruent to $1$ in the generic points of $X_k$, i.e. in each $\mathcal{O}_{X,{\mu}}$, and $Div(X,X_k)$ is the sheaf associated to $\mathcal{M}^*_{X;X_k}/\mathcal{O}^*_{X;X_k}$. In other words, $Div(X,X_k)(U)$ is the set of relative Cartier divisors on $U\subset X$ which restrict to zero in $Z_0(X_k)$. For the concept of relative meromorphic functions and divisors see \cite[Sec. 20, 21.15]{EGA4}.
We want to show that $Div(X,X_k)(X)/\mathcal{M}^*_{X;X_k}(X)=0$. The short exact sequence (\ref{eq1}) implies that $Div(X,X_k)(X)/\mathcal{M}^*_{X;X_k}(X)$ injects into $\Pic(X,X_k)$. But the latter group fits into the exact sequence 
$$H^0(X,\roi_X^\times)\twoheadrightarrow H^0(X_k,\roi_{X_k}^\times)\to \Pic(X,X_k)\to \Pic(X)\xto{\cong} \Pic(X_k)$$
in which the first map is surjective, since it is possible to lift units, which can be shown using the Stein factorization. Furthermore, the map on the right is an isomorphism by Proposition \ref{proposition_curve_case}. Therefore $\Pic(X,X_k)=0$.

Let $d>1$. As noted in the beginning of the proof, it suffices to show that for $x\in \ker[res:Z_1(X) \to Z_0(X_k)]$ there is a $\xi\in \ker[res:C_2(X,-1)\to C_1(X_k,0)]$ with $\partial(\xi)=x$. For the reduction to the case $d=1$ we now proceed as in the proof of Theorem \ref{theorem_lifting_well_defined} with the following modifications. 
We embed $\tilde Z_{\roi_{Z}}$ and $\tilde Z'_{\roi_{Z}}$ as $\tilde Z_{\roi_{Z}}=(Z_{\roi_{Z}},(1:0),(1:0))$ and $\tilde Z'_{\roi_{Z}}=(Z'_{\roi_{Z}},(1:0),(1:0))$ into $X\times \bb P^2_{\roi_Z}$. In particular they intersect the special fiber in the same point, say $p$. As in \textit{Step 3} of the proof of Theorem \ref{theorem_lifting_well_defined} we can find relative curves of genus zero $C_S$ and $C_S'$ containing $\tilde Z_{\roi_{Z}}$ and $\tilde Z'_{\roi_{Z}}$ respectively with the property that $C_S$ and $C_S'$ intersect transversally in $p$. Let $\tilde Z_{\roi_{Z}}''$ the intersection of the two relative curves at $p$. Then we apply the case $d=1$ proved above to $\tilde Z_{\roi_{Z}}-\tilde Z_{\roi_{Z}}''$ and $\tilde Z'_{\roi_{Z}}-\tilde Z_{\roi_{Z}}''$ in $C_S$ and $C_S'$ respectively.
\end{proof}

\begin{remark}
Let $j:X_K\to X$ denote the inclusion of the generic fiber. Again, the group $\CH^d(X,1)$ links the groups $\CH^d(X_K,1)$ and $\CH^d(X_k,1)$:
$$\xymatrix{
& \CH^d(X,1)  \ar@{=}[d] \ar[r]^{res^{d,1}_{\bb Z}} & \CH^d(X_k,1) & \\
\CH^{d-1}(X_0,1) \ar[r] & \CH^d(X,1)    \ar[r]^{j^*} &  \CH^d(X_K,1)      \ar[r]  & \CH^{d-1}(X_k,0)=\CH_1(X_k). \
}$$
The importance of the group $\CH^d(X,1)$ comes from the fact that it is related to the torsion in $\CH^d(X_K)$ (see \cite[Sec. 3]{Lu16}).
\end{remark}

\begin{corollary}
Let the notation be as in Theorem \ref{theorem_higher_zero}. Assume that the Gersten conjecture holds for the Milnor K-sheaf $\K^M_{d,X}$.
Then $$H^d(X,\K^M_{d,X\mid X_k})=0.$$
\end{corollary}
\begin{proof}
The short exact sequence of
sheaves on $X$
$$0\to \K^M_{d,X\mid X_k}\to \K^M_{d,X} \to \K^M_{d, X_k} \to 0$$
induces an exact sequence 
$$\to H^{d-1}(X,\K^M_{d,X}) \to H^{d-1}(X,\K^M_{d, X_k})  \to H^d(X,\K^M_{d,X\mid X_k}) \to H^d(X,\K^M_{d,X}) \to H^d(X,\K^M_{d, X_k}) \to 0$$
which, by the remarks at the beginning of the section, is isomorphic to
$$\to \CH^{d}(X,1) \twoheadrightarrow \CH^{d}(X_k,1)  \to H^d(X,\K^M_{d,X\mid X_k}) \to \CH^d(X) \xto{\cong} \CH^d(X_k) \to 0.$$
The statement therefore follows from Theorem \ref{main_theorem_in_text} and Theorem \ref{theorem_higher_zero}.
\end{proof}


\section{Conjectures in the non-smooth case and the case of rational surfaces}\label{section_d=2}
In this section we study the case of smooth rationally connected varieties over local fields which do not have good reduction. More precisely, we would like to understand how the Chow group of $1$-cycles of a regular model of such a variety over the ring of integers relates to the Chow group of zero cycles of the generic fiber and to a cohomological version of the Chow group of zero cycles of the special fiber. One motivation is to use finiteness results about the latter to deduce the following conjecture:
\begin{conjecture}[Koll\'ar-Szab\'{o}]\label{conjecture_Kollar_Szabo}
Let $X_K$ be a $d$-dimensional, smooth, projective, separably rationally connected variety defined over a local field. Then $\CH^d(X_K)_0:=\ker[deg:\CH_0(X_K)\to \bb Z]$ is finite.
\end{conjecture}
Conjecture \ref{conjecture_Kollar_Szabo} is known to be true if $X_K$ has good reduction. This is due to Koll\'ar and Szab\'{o} \cite[Thm. 5]{Kollar2003}. They even show that $\CH^d(X_K)_0=0$. This also follows from the main theorem of \cite{Kollar2004} or our Corollary \ref{corollary_main_thm_in_text}, combined with the following theorem:
\begin{theorem}(Kato-Saito \cite[Thm. 1]{KaS83})
If $k$ is a finite field and $X_k$ a $d$-dimensional, smooth, projective, separably rationally connected variety over $k$, then $\CH^d(X_k)_0=0$.
\end{theorem}
In general, it cannot be expected that $\CH^d(X_K)_0=0$. In \cite[Thm. 1.1]{SS14}, Saito and Sato calculate this group for certain cubic surfaces to be isomorphic to $\bb Z/3\oplus \bb Z/3$. Nevertheless, Conjecture \ref{conjecture_Kollar_Szabo} holds in dimension two:
\begin{theorem}(Colliot-Th\'el\`ene \cite[Thm. A]{CT1983})\label{theorem_CT_surfaces}
Let $X_K$ be a smooth projective rational surface over a local field $K$. Then $\CH^2(X_K)_0$ is finite.
\end{theorem}

From now on let $A$ be a henselian discrete valuation ring with local parameter $\pi$, residue field $k$ and quotient field $K$. Let $X$ be a regular scheme flat and projective over $A$ with separably rationally connected generic fiber $X_K$. Assume that the special fiber $X_k$ (in the following also denoted by $X_1$) of $X$ is a simple normal crossing divisor. We denote the respective inclusions by $j:X_K\to X$ and $i:X_k\to X$ and set $X_n=X\times_A A/\pi^n$. Let $d$ be the relative dimension of $X$ over $\Spec(A)$. As mentioned above, we intend to study Conjecture \ref{conjecture_Kollar_Szabo} for $\CH^d(X_K)$ using different (cycle) class maps out of $\CH^d(X)$. We begin with a few remarks on these maps.
Assuming the Gersten conjecture for Milnor K-theory, there is an isomorphism
$$\CH^d(X)/p^r\cong H^d(X,\hat{\cal K}^M_{d,X}/p^r) $$
which, by composition, allows us to define a restriction map
$$res^d_{/p^r}:\CH^d(X)/p^r\cong H^d(X,\hat{\cal K}^M_{d,X}/p^r) \to  \projlim_n H^{d}(X_k,\hat{\cal K}^M_{d,X_n}/p^r).$$
This is the Zariski side of the story. If $K$ is a $p$-adic local field, the right $p$-adic \'etale motivic theory can be defined by Sato's $p$-adic \'etale Tate twists, which are objects $\cal T_r(n)\in D^b(X_{\et},\bb Z/p^r)$ fitting into a distinguished triangle of the form
$$i_*\nu^{n-1}_{Y,r}[-n-1]\xto{} \cal T_r(n)_X\xto{}\tau_{\leq n}Rj_*\mu_{p^r}^{\otimes n}\xto{}i_*\nu^{n-1}_{Y,r}[-n]$$
(see \cite{Sato2007}).
The Galois symbol map $\hat{\cal K}^M_{d,X}/p^r\to \cal H^d(\cal T_r(d)_X)$ induces
a map $cl_X:H^d(X,\hat{\cal K}^M_{d,X}/p^r)\to H^{2d}(X,\cal T_r(d))$.\footnote{The map $cl_X:\CH^d(X)/p^r\to H^{2d}(X,\cal T_r(d))$ can also be defined without using the Gersten conjecture \cite[Cor. 6.1.4]{Sato2007}.} 
The map $i^*\hat{\cal K}^M_{d,X}/p^r\to i^*\cal H^d(\cal T_r(d)_X)$ factorises through $\K^M_{d,X_n}/p^r$ for $n$ large enough. Via a norm argument one first reduces this statement to the case in which the residue field is large and therefore the improved Milnor K-theory coincides with ordinary Milnor K-theory. This case can be deduced from the short exact sequence $0\to 1+\pi^n\roi_X \to \roi_X^\times \to \roi_{X_n}^\times \to 0$, which induces the unit filtration on Milnor K-theory, noticing that in the Nisnevich topology the group $(1+\pi^n\roi_X)/p^r=0$ for $n$ large enough by Hensel's lemma \cite[II, Lem. 2]{Elkik1973}. 
Indeed, the kernel of the restriction map $i^*\cal K^M_{d,X}\to \cal K^M_{d,X_n}$ is locally generated by elements of the form $\{1+\pi^n a,f_1,\dots,f_{d-1}\}$ with $f_1,\dots,f_{d-1}$ units \cite[Lemma 1.3.1]{KaS86}. In \cite{Lueders2019} we showed that the resulting map $\projlim_n H^{d}(X_k,\K^M_{d,X_n}/p^r)\to H^{2d}(X, \cal T_r(d))$ is an isomorphism in the smooth case. 
In sum, there are commutative diagrams
$$
\xymatrix{
\CH^d(X)/p^r \ar[r]^-{res^d_{/p^r}} \ar@{->>}[d]_{cl_X} & \projlim_n H^{d}(X_k,\K^M_{d,X_n}/p^r) \ar[d]^{cl_{X_k}} & \mathrm{and} & \CH^d(X)/p^r \ar[r]^-{res^d_{/p^r}} \ar@{->>}[d]_{cl_X} & \projlim_n H^{d}(X_k,\K^M_{d,X_n}/p^r)  \ar[dl]^{}  \\
H^{2d}(X, \cal T_r(d)) \ar[r]^-\cong & H^{2d}(X_k,i^*\cal T_r(d)) & & H^{2d}(X, \cal T_r(d)).  &  \\
}
$$
The bottom map on the left is an isomorphism by \'etale proper base change. The map $cl_X$ is surjective by \cite{SS14}. 
Precomposing the vertical maps with maps from $p^r$-torsion subgroups, we get a commutative diagram
\begin{equation}\label{diagram_padic}
\xymatrix{
\CH^d(X)[p^r] \ar@{-->}[r]^-{res^d|_{tor}} \ar[d]_\sigma &  \projlim_n H^{d}(X_k,\K^M_{d,X_n})[p^r]  \ar[d]^{\sigma_k}  \\
\CH^d(X)/p^r \ar[r]^-{res^d_{/p^r}} \ar@{->>}[d]_{cl_X} & \projlim_n H^{d}(X_k,\K^M_{d,X_n}/p^r) \ar[d]^{cl_{X_k}} \\
H^{2d}(X, \cal T_r(d)) \ar[r]^\cong & H^{2d}(X_k,i^*\cal T_r(d)). \\
}
\end{equation}

\begin{remark}\label{remark_KEW}
\begin{enumerate}
\item In \cite[Conj. 10.1]{KerzEsnaultWittenberg2016}, Kerz, Esnault and Wittenberg conjecture that, assuming the Gersten conjecture for Milnor K-theory for the isomorphism on the left, the map $$res^d_{/p^r}:\CH^d(X)/p^r\cong H^{d}(X,\cal K^M_{d,X}/p^r) \to \projlim_n H^{d}(X_k,\K^M_{d,X_n}/p^r)$$ is an isomorphism for any regular flat projective scheme over $A$. In \cite{Lueders2019} we showed that this map is surjective in the smooth case even with integral coefficients. 
\item 
We would like to study the map $res^d|_{tor}$ using the map $res^d_{/p^r}$. For this the map $\sigma$ needs to be injective. Then we can deduce the injectivity of $res^d|_{tor}$ from the conjectural injectivity of $res^d_{/p^r}$. If $\CH^d(X)[p^r]$ is of finite exponent, then $\sigma$ is injective for $r>>0$.
\item If $p\neq ch(k)$, then by Corollary \ref{corollary_App_B} the above diagram becomes the following diagram:
\begin{equation}\label{diagram_ladic}
\xymatrix{
\CH^d(X)[\ell^r] \ar@{-->}[r]^-{res^d|_{tor}} \ar[d]_\sigma &  H^{d}_{}(X_k,\K^M_{d,X_k})[\ell^r]  \ar[d]^{\sigma_k}  \\
\CH^d(X)/\ell^r \ar[r]^-{res^d_{/\ell^r}} \ar[d]_{cl_X} & H^{2d}_{\mathrm{cdh}}(X_k,\bb Z/\ell^r(d)) \ar[d]^{cl_{X_k}} \\
H^{2d}(X,\mu_{\ell^r}^{\otimes d}) \ar[r]^\cong & H^{2d}(X_k,\mu_{\ell^r}^{\otimes d}). \\
}
\end{equation}
For the notation and definition of the cdh-cohomology group on the right see Section \ref{subsection_cdh–motivic}. In this case the map $res^d_{/\ell^r}$ is known to be an isomorphism by the main result of \cite{KerzEsnaultWittenberg2016} assuming $k$ is finite or separably closed.
\end{enumerate}
\end{remark}


Saito and Sato prove the following important theorem concerning the composition $cl_X\circ \sigma= cl_{X,p\mathrm{-tors},r}$.

\begin{theorem}(\cite[Cor. 4.3]{Sato2005},\cite[Thm. 1.5]{SaitoSato2010})\label{theorem_SaSa_rational_surf} Let $A$ be a henselian discrete valuation ring with perfect residue field $k$ and quotient field $K$. Let $X$ be a regular scheme flat and projective over $A$ with generic fiber $X_K$. Let $d$ be the relative dimension of $X$ over $A$.
Assume that $H^2(X_K,\roi_{X_K})=0$. Let $ch(k)=p$, $ch(K)=0$ and assume that the special fiber is a simple normal crossing divisor. Then the map $$cl_{X,p\mathrm{-tors},r}:\CH^2(X)_{p\mathrm{-tors}}\to H^{4}(X,\cal T_r(2))$$ is injective for sufficiently large $r$.
\end{theorem}

\begin{remark}
The condition $H^2(X_K,\roi_{X_K})=0$ is satisfied if $X_K$ is a smooth projective separably rationally connected variety of dimension $2$. In this case, by \cite[Ch. IV, Cor. 3.8]{Kollar1996}, $H^0(X_K,\Omega^m_{X_K})=0$ for all $m>0$ and by Serre duality, $H^0(X_K,\Omega^2_{X_K})=H^2(X_K,\roi_{X_K})$.
\end{remark}

\begin{proposition}\label{proposition_torsion_fe}
Let $K$ be a local field with residue field $k$ and ring of integers $A$. Let $X$ be a regular flat projective scheme over $A$. Assume that the generic fiber $X_K$ of $X$ is rationally connected of dimension $2$, i.e. a rational surface, then the group $\CH^2(X)$ is finitely generated and the group
$\CH^2(X)_{\mathrm{tors}}$ 
is finite.
\end{proposition}
\begin{proof}
The group $\CH^2(X)$ fits into the exact sequence
$\CH^{1}(X_k)\to \CH^2(X) \to \CH^2(X_K)\to 0.$
By Theorem \ref{theorem_CT_surfaces}, $\CH^2(X_K)$ is the direct sum of $\bb Z$ and a finite group. The group $\CH^{1}(X_k)$ is finitely generated since the Picard group of each component of $X_k$ is finitely generated since $k$ is finite. Therefore $\CH^2(X)_{\mathrm{tors}}$ is finite.
\end{proof}

\begin{theorem}\label{theorem_surfaces}
Let the notation be as in Proposition \ref{proposition_torsion_fe} and $ch(k)=p$ and $ch(K)=0$. We make the following additional assumptions:
\begin{enumerate}
\item The reduced special fiber $X_k$ is a simple normal crossing divisor.
\item The Gersten conjecture holds for the Milnor K-sheaf $\K^M_{2,X}$ and the induced map $ \CH^2(X)\cong H^{2}(X,\K^M_{2,X})  \to \projlim_n H^{2}(X_k,\K^M_{2,X_n})$ is surjective.
\end{enumerate}
Then the map
$$res^2: \CH^2(X) \to \projlim_n H^{2}(X_k,\K^M_{2,X_n})$$
is an isomorphism.
\end{theorem}
\begin{proof}
By Theorem \ref{theorem_SaSa_rational_surf} and the  commutativity of Diagram (\ref{diagram_padic}) the map $res^2$ is injective on $p$-torsion. By Proposition \ref{proposition_torsion_fe} and and the commutativity of Diagram (\ref{diagram_ladic}) the map $res^2$ is injective on $\ell$-torsion. Combined we get an injection
$$res^2:\CH^2(X)_{\mathrm{tors}} \hookrightarrow \projlim_n H^{2}(X_k,\K^M_{2,X_n})_{\mathrm{tors}}.$$
Since by Proposition \ref{proposition_torsion_fe} the group $\CH^2(X)$ is finitely generated, the theorem now follows combining the injectivity on the torsion subgroup with assumption (ii) and the fact that the map $res^2_{/\ell^r}$ is an isomorphism for all $\ell\neq p$ by the main theorem of \cite{KerzEsnaultWittenberg2016}.
\end{proof}


\begin{conjecture}\label{conjecutre_main}
Let $A$ be a henselian discrete valuation ring with residue field $k$ and quotient field $K$. Let $X$ be a regular scheme flat and projective over $A$ with separably rationally connected generic fiber $X_K$. Let $d$ be the relative dimension of $X$ over $\Spec(A)$. Assume the Gersten conjecture for Milnor K-theory. Then the restriction map 
$$res^d: CH^{d}(X)\cong H^{d}(X,\K^M_{d,X}) \to \projlim_n H^{d}(X_k,\K^M_{d,X_n})$$
is an isomorphism.
\end{conjecture}

\begin{remark}
Conjecture \ref{conjecutre_main} should be viewed as an analogue of the conjecture of Kerz, Esnault and Wittenberg described in Remark \ref{remark_KEW}(ii) and as a generalisation of Theorem \ref{theorem_surfaces}. 
If $\CH^d(X_K)$ is $p$-torsion free, then we expect that  $\projlim_n H^{d}(X_k,\K^M_{d,X_n})\cong H^{d}(X_k,\K^M_{d,X_k})\cong \CH_0^{LW}(X_k)$.
\end{remark}


\appendix
\section{The Gersten conjecture}\label{appendix_Gersten_higher_chow}
In this appendix we recall a few facts about the Gersten conjecture which we need in the article.

\paragraph{The Gersten conjecture for Milnor K-theory.}
Let $X$ be a regular scheme. Let $\hat{K}^M_n$ denote the improved Milnor K-theory defined in \cite{Kerz2010} and $\hat{\mathcal{K}}^M_{n,X}$ its sheafification. $\hat{K}^M_n$ coincides with usual Milnor K-theory for fields and local rings with big residue fields. The Gersten conjecture for Milnor K-theory says that the sequence of sheaves
$$0\to \hat{\mathcal{K}}^M_{n,X}\xrightarrow{i} \bigoplus_{x\in X^{(0)}}i_{x*}K^M_n(x)\rightarrow \bigoplus_{x\in X^{(1)}}i_{x*}K^M_{n-1}(x)\to\dots \to\bigoplus_{x\in X^{(d)}}i_{x*}K^M_{n-d}(x)\to 0$$
is exact. 
\begin{proposition}\label{Gersten_KM}
The Gersten conjecture for Milnor K-theory holds in the following cases:
\begin{enumerate}
\item If $X$ is a regular scheme over a field.
\item If $X=\Spec(R)$ for $R$ a henselian discrete valuation ring with finite residue field.
\end{enumerate}
\end{proposition}
\begin{proof}
(i) is shown in \cite{Kerz2010}, (ii) is shown in \cite{Da18} and \cite[Thm. 5.1]{LuedersGerstenKM}.
\end{proof}

\paragraph{The Gersten conjecture for higher Chow groups for DVRs.} Let $R$ be a discrete valuation ring with residue field $k$ and fraction field $K$. The Gersten conjecture for higher Chow groups in this case says that the localisation exact sequence
$$\dots\to \CH^r(K,q+1)\to \CH^{r-1}(k,q)\to \CH^r(R,q)\to \CH^r(K,q)\to \CH^{r-1}(k,q-1)\to\dots,$$
splits into the pieces
$$0\to \CH^r(R,q)\to \CH^r(K,q)\to \CH^{r-1}(k,q-1)\to 0.$$ 

\begin{proposition}\label{Gersten_higherChow}
Let $r\geq 0$. Assume that either
\begin{enumerate}
\item $R$ is a discrete valuation ring containing a field, or
\item  $R$ is a henselian discrete valuation ring with finite residue field and $r\leq 3$. 
\end{enumerate}
Then there is an isomorphism
$$s:\hat K^M_r(R)\to \CH^{r}(R,r)$$
and the Gersten conjecture holds for the higher Chow groups of $R$ in all degrees in the case of (i) and in the Milnor range $(r,r)$ for $r\leq 3$ in the case of (ii).
\end{proposition}
\begin{proof}
(i) If $R$ contains a field, then the Gersten conjecture is known to hold by \cite[Sec. 10]{Bl86} and Panin's trick \cite{Pa03}.

(ii) We have that 
$$\CH^{r-1}(k,r)\cong H^{r-2.r-1}(k,\bb Z)=0$$
(for the isomorphism on the left of higher Chow groups with motivic cohomology see for example \cite[Thm. 19.1]{MVW06}) for $r\leq 2$. Indeed, this is so by definition for $r=0$. For $r=1$ this is implied for example by the Beilinson-Soul\'e vanishing conjecture which holds for finite fields (see \cite[Ch. VI, Ex. 4.6, p. 535]{Weibel2013}). For $r=2$ this can be deduced from the motivic-to-K-theory spectral sequence
$$E_2^{p,q}=\CH^{-q}(k,-p-q)\Rightarrow K_{-p-q}(k)$$
and the fact that $K^M_2(k)=K_2(k)=0$. 
We conclude that under the assumptions of the proposition, there is a commutative diagram with exact sequences and vertical isomorphisms
$$\xymatrix{
0 \ar[r]^{} & \CH^{r}(R,r) \ar[r]  & \CH^{r}(K,r)  \ar[r]   & \CH^{r-1}(k,r-1)   \ar[r]^{}  & 0 \\
0 \ar[r]^{} & \hat K^M_r(R) \ar[r]^{}\ar[u]^{\cong}  & K^M_r(K) \ar[r]^{}\ar[u]^{\cong}   & K^M_{r-1}(k)  \ar[r]^{} \ar[u]^{\cong}   & 0 .   \
}$$
Indeed, $s$ is always an isomorphism for fields by \cite{To92} and \cite{NS89}. The lower horizontal sequence is exact by Proposition \ref{Gersten_KM}(ii) and therefore $s$ is always injective. By the above remarks the upper horizontal sequence is exact for $r\leq 2$. For $r=3$ the motivic-to-K-theory spectral sequence implies that $\CH^2(k,3)\cong K_3(k)$ is finite by Quillen's calculation of the K-groups of finite fields. The result then follows from the fact that $s$ is an isomorphism with finite coefficients.
\end{proof}


\section{Cohomological Chow groups of zero-cycles}\label{section_coh_chow}
Let $X$ be a scheme of finite type over a field $k$. The G-theory $G_n$ of $X$, defined using the abelian category of coherent sheaves on $X$, is related to Bloch's higher Chow groups via a version of the Grothendieck-Riemann-Roch theorem for singular schemes; this theorem is due to Baum-Fulton-MacPherson \cite{Baum1975} and Bloch \cite{Bl86} and says that there is a sequence of isomorphisms
$$\tau:\bigoplus_i\CH^i(X,n)_{\bb Q}\xto{\cong}\bigoplus_i \gr^i G_n(X)_{\bb Q}\xto{\cong} G_n(X)_\bb Q. $$
This is considered to be the homological side of the story since Bloch's higher Chow groups form a Borel-Moore homology theory. 

On the cohomological side there are two versions of K-theory. Algebraic K-theory, defined by Quillen using the category of locally free sheaves on $X$, and $KH$-theory, a homotopy invariant version of algebraic K-theory, defined by Weibel. Algebraic K-theory should be related to a cohomological version of Bloch's higher Chow groups. In the case of zero-cycles Levine and Weibel \cite{LeWe85} have defined a version of the Chow group of zero-cycles on $X$ in terms of the smooth points and certain good curves on $X$, which is related to $K_0(X)$ (for a precise definition see Definition \ref{definition_Chow_group_LW}). 

Finally, $KH$-theory is related to cdh-motivic cohomology. If $A$ is a ring and $\ell\in A^\times$, then there is an isomorphism $K_n(A,\bb Z/\ell\bb Z)\cong KH_n(A,\bb Z/\ell\bb Z)$ (see \cite[Prop. 1.6]{Weibel1989}). This indicates that with finite coefficients cohomological Chow groups should be closely related to cdh-motivic cohomology groups.

In Section \ref{subsection_curve_case} and \ref{subsection_arb_dim} we define cohomological versions of the higher zero-cycles $\CH^{d+j}(X,j)$ which we have encountered in Section \ref{section_Filtrations on higher zero-cycles}. This extends the definition of Levine and Weibel for zero-cycles and is related to higher K-theory. In Section \ref{subsection_cdh–motivic} we relate these groups to cdh-motivic cohomology. For more results of this nature we refer the reader to \cite{BindaKrishna2022}. 

\subsection{The case of curves}\label{subsection_curve_case}
In this and the next subsection we assume for simplicity that all residue fields are large in order to be able to work with ordinary Milnor K-theory. In the following let $k$ be a field and $X$ be an equidiensional quasi-projective $k$-scheme.
Let $X_{\rm sm}$ denote the disjoint union of the smooth loci of the   
$d$-dimensional irreducible components of $X$. Let $X_{\rm sing}$ denote the singular locus of $X$. 
A smooth closed point of $X$ will mean a closed point lying in  
$X_{\rm sm}$.  
Let $Y \subsetneq X$ be a closed subset not containing any $d$-dimensional  
component of $X$ and such that $X_{\rm sing} \subseteq Y$.   
Let $Z_0(X,Y)$ be the free abelian group on closed points of $X \setminus Y$.  
We shall sometimes write $Z_0(X,X_{\rm sing})$ as $Z_0(X)$.
More generally, we make the following definition:
\begin{definition}
$$Z_0(X,X_{\rm sing},j):=\bigoplus_{x\in Z_0(X,X_{\rm sing})}K^M_j(x).$$
\end{definition}

Let $C$ be a reduced scheme of pure dimension $1$ over $k$ and $\{\eta_1, \cdots , \eta_r\}$ be the set of generic points of $C$.   
Let $\cal O_{C,Z}$ denote the semi-local ring of $C$ at   
$S = Z \cup \{\eta_1, \cdots , \eta_r\}$.  
Let $k(C)$ denote the ring of total  
quotients of $C$ and write $\cal O_{C,Z}^\times$ for the group of units in   
$\cal O_{C,Z}$. Notice that $\cal O_{C,Z}$ coincides with $k(C)$   
if $|Z| = \emptyset$.   
Since $C$ is reduced, it is Cohen-Macaulay and therefore $\cal O_{C,Z}^\times$  is the subgroup of group of units in the ring of total quotients $k(C)^\times$   
consisting of those $f$ which are regular and invertible in the local rings $\cal O_{C,x}$ for every $x\in Z$. 
\begin{definition}
We let $$K^M_j(\cal O_{C,Z}):=(\cal O_{C,Z}^\times)^{\otimes j}/<a_1\otimes\dots\otimes a_j\mid a_i+a_{i'}=1 \;\rm for\; some\; 1\leq i<i'\leq j>.$$
\end{definition}
By definition, if $\alpha\in K^M_{j+1}(\cal O_{C,Z})$, then $\partial(\alpha)\in Z_0(C,C_{\rm sing},j).$
\begin{definition}
We let $R_0(C,j)$ be the subgroup of $Z_0(C,C_{\rm sing},j)$ generated by all $\partial(\alpha)$ for $\alpha\in K^M_{j+1}(\cal O_{C,Z})$.
\end{definition}

\begin{proposition}\label{Bloch_formula_dim_1}
Let $X$ be of dimension $1$.
Then $$H^1(X,\cal K^M_{j+1})\cong Z_0(X,X_{\rm sing},j)/R_0(C,j).$$
\end{proposition}
\begin{proof}
We consider the local to global spectral sequence
$$E_1^{p,q}(X,\cal K^M_{j+1})=\bigoplus_{x\in X^{(p)}}H_x^{p+q}(X,\cal K^M_{j+1})\Longrightarrow H^{p+q}(X,\cal K^M_{j+1}).$$
Since $H^1_y(X,\cal K^M_{j+1})=H^1(k(y),\cal K^M_{j+1})=0$ for $y\in X^{(0)}$, the above spectral sequence implies that
$$H^1(X,\cal K^M_{j+1})\cong\mathrm{coker}(\partial:\bigoplus_{y\in X^{(0)}}H_y^{0}(X,\cal K^M_{j+1})\to \bigoplus_{x\in X^{(1)}}H_x^{1}(X,\cal K^M_{j+1})).$$
Purity for Milnor K-theory (see \cite[Lem. 3.4]{Lueders2019}) implies that if $y\in X^{(0)}_{}$, then $H_y^{0}(X,\cal K^M_{j+1})\cong K^M_{j+1}(y)$ and if $x\in X^{(1)}_{\rm reg}$, then $H_x^{1}(X,\cal K^M_{j+1}))\cong K^M_j(x).$ Let $x\in X^{(1)}_{\rm sing}$ and $\frak m$ be the maximal ideal of $\roi_{X,x}$. Then the exact sequence 
$$H^0(\roi_{X,x},\cal K^M_{j+1})\to H^0(\roi_{X,x}-\frak m,\cal K^M_{j+1})\to H^1_x(\roi_{X,x},\cal K^M_{j+1})\to H^1(\roi_{X,x},\cal K^M_{j+1})=0$$
implies that 
\begin{equation}\label{equation_local_coh_Milnor_K}
 H_x^{1}(X,\cal K^M_{j+1}))\cong H^0(\roi_{X,x}-\frak m,\cal K^M_{j+1})/H^0(\roi_{X,x},\cal K^M_{j+1}).
\end{equation} 
Therefore the map $\bigoplus_{y:x\in \overline{\{y\}}}H_y^{0}(X,\cal K^M_{j+1})\cong K^M_{j+1}(y)\to H_x^{1}(X,\cal K^M_{j+1})$ is surjective. Furthermore, using a moving lemma \cite[Lem. 9.1.9]{Li02}, one can show that the map $$\bigoplus_{y\in X^{(0)}}H_y^{0}(X,\cal K^M_{j+1})\to \bigoplus_{x\in X_{\rm sing}^{(1)}}H_x^{1}(X,\cal K^M_{j+1})$$
is surjective and therefore $H^1(X,\cal K^M_{j+1})$ is generated by $\bigoplus_{x\in Z_0(X,X_{\rm sing})}K^M_j(x).$ Finally, (\ref{equation_local_coh_Milnor_K}) implies that the kernel of the map 
$$\bigoplus_{x\in Z_0(X,X_{\rm sing})}K^M_j(x)\to H^1(X,\cal K^M_{j+1})$$
is given by all elements of $\bigoplus_{y\in X^{(0)}}H_y^{0}(X,\cal K^M_{j+1})$ which restrict to $H^0(\roi_{X,x},\cal K^M_{j+1})$ in $H_x^{1}(X,\cal K^M_{j+1})$ for $x\in X^{(1)}_{\rm sing}$. 
\end{proof}

\begin{example}[Cusp]
Let $R$ be a ring and $I\subset R$ be an ideal. Let $f:R\to S$ be a ring homomorphism that maps $I$ isomorphically onto an ideal of $S$, which we also denote by $I$. In this case $R$ is the pullback of $S$ and $R/I$ and the pullback square
$$\xymatrix{
R \ar[r]^{f} \ar[d]^{} & S \ar[d]^{} \\
R/I \ar[r]^{\bar{f}}  & S/I,    \
}$$
is called a Milnor square. We consider this square for the cusp $R=k[x,y]/(x^3=y^2),  S=k[t], I=t^2S, R/I=k$ and $S/I=k[t]/t^2$ and $f$ maps $x$ to $t^2$ and $y$ to $t^3$. Let $P$ be the origin. Then, using the isomorphisms $$H^1_P(\Spec(R),\cal K^M_{j+1})\cong H^0(\Spec(R_P)-P,\cal K^M_{j+1})/H^0(\Spec(R_P),\cal K^M_{j+1})$$ and $$H^1_P(\Spec(S),\cal K^M_{j+1})\cong H^0(\Spec(S_P)-P,\cal K^M_{j+1})/H^0(\Spec(S_P),\cal K^M_{j+1}),$$ there is a short exact sequence
$$0\to H^0(\Spec(S_P),\cal K^M_{j+1})/H^0(\Spec(R_P),\cal K^M_{j+1})\to H^1_x(\Spec(R),\cal K^M_{j+1})\to H^1_x(\Spec(S),\cal K^M_{j+1})\to 0.$$ 
We denote the term on the left by $\Omega^j$. For $j=0$ it is isomorphic to $k$. We get a short exact sequences of two-term complexes
$$\xymatrix{
& 0\ar[r] \ar[d] &  K^M_{j+1}(k(t)) \ar[r] \ar[d] &    K^M_{j+1}(k(t))  \ar[d]^{}  \ar[r] & 0 \\
 0 \ar[r] &  \Omega^j \ar[r]  & \bigoplus_{x\in \Spec(R)_{(1)}}H^1_x(\Spec(R),\cal K^M_{j+1})\ar[r]  &   \bigoplus_{x\in \Spec(S)_{(1)}}H^1_x(\Spec(S),\cal K^M_{j+1})  \ar[r] & 0.  \
}$$
This induces a long exact sequence 
$$A_1(\Spec(R),j)\xto{\cong} A_1(\Spec(S),j)\to \Omega\xto{} H^1(\Spec(R),K^M_{j+1})\to H^1(\Spec(S),K^M_{j+1})=0,$$
in which the first map is an isomorphism by the definition of the tame symbol and the last group is zero by homotopy invariance. Therefore $\Omega{\cong} H^1(\Spec(R),K^M_{j+1})$ which for $j=0$ recovers the well known isomorphism $k\cong \Pic(\Spec(R))$ (see \cite[Ch. I., Ex. 3.10.1]{Weibel2013}).
\end{example}

\subsection{In arbitrary dimension}\label{subsection_arb_dim}
We keep the notation of the beginning of Section \ref{subsection_curve_case}. We begin by recalling the definition of the Binda-Krishna Chow group of zero-cycles (or short Binda-Krishna Chow group) from \cite[Sec. 3]{Binda2018}.  

\begin{definition}\label{definition_good_curve}  
	Let $C$ be a reduced scheme which is of pure dimension one over $k$.  
	We shall say that a pair $(C, Z)$ is \emph{a good curve  
		relative to $X$} if there exists a finite morphism $\nu\colon C \to X$  
	and a  closed proper subscheme $Z \subsetneq C$ such that the following hold.  
	\begin{enumerate}  
		\item  
		No component of $C$ is contained in $Z$.  
		\item  
		$\nu^{-1}(X_{\rm sing}) \cup C_{\rm sing}\subseteq Z$.  
		\item  
		$\nu$ is local complete intersection at every   
		point $x \in C$ such that $\nu(x) \in X_{\rm sing}$.   
	\end{enumerate}  
\end{definition}

Let $(C, Z)$ be a good curve relative to $X$. 
Since $\nu$ is proper and maps $Z_0(C,Z)$ to $Z_0(X,X_{\rm sing})$, there exists a   
push-forward map $$\nu_{*}:Z_0(C,Z,j)\xrightarrow{} Z_0(X,X_{\rm sing},j).$$  

\begin{definition}(Binda-Krishna Chow group)  
Let $R_0(C, Z, X,j)$ be the subgroup  
of $Z_0(X,X_{\rm sing},j)$ generated by the set   
$\{\nu_*(\partial(\alpha))| \alpha \in K^M_{j+1}(\cal O^{}_{C, Z})\}$.   
Let $R_0(X,j)$ denote the subgroup of $Z_0(X,X_{\rm sing},j)$ generated by   
the image of the map $R_0(C, Z, X,j) \to Z_0(X,X_{\rm sing},j)$, where  
$(C, Z)$ runs through all good curves relative to $X$. 

The quotient group ${Z_0(X,X_{\rm sing},j)}/{R^{}_0(X,j)}$ is  
denoted by $\CH^{BK}_0(X,j)$ and called the Binda-Krishna Chow group. 
\end{definition}

We can now define the Levine-Weibel Chow group.
\begin{definition}(Levine-Weibel Chow group)\label{definition_Chow_group_LW} 
Let $R^{LW}_0(X,j)$ denote the subgroup of $Z_0(X,X_{\rm sing},j)$ generated  
by the divisors of rational functions on good curves $(C,Z)$ relative to $X$ but assume additionally that the map $\nu: C \to X$ is a closed immersion. Such curves on $X$ are called {\sl Cartier curves}.   
The quotient group ${Z_0(X,X_{\rm sing},j)}/{R^{LW}_0(X,j)}$ is  
denoted by $\CH^{LW}_0(X,j)$ and called the Levine-Weibel Chow group. We often write $\CH^{LW}_0(X)$ for $\CH^{LW}_0(X,0)$
\end{definition}

There is a canonical surjection  
$\CH^{LW}_0(X) \to \CH^{BK}_0(X)$ (see \cite[Lem. 3.10]{Binda2018}). The definition of the Levine-Weibel Chow group \cite{LeWe85} predates the definition of the Binda-Krishna Chow group \cite{Binda2018}.  The Binda-Krishna Chow group is more flexible and easier to handle since it often allows to work with curves which are regularly embedded in $X$ (see \cite[Lem. 3.5]{Binda2018}).
It is not known if the Levine-Weibel and Binda-Krishna Chow group coincide in general, unless 
$X$ is a reduced quasi-projective scheme of dimension $d\ge 1$ over 
	an algebraically closed field $k$ and either (i)
		$d \le 2$, (ii)
		$X$ is affine, or (iii)
		${\rm char}(k) = 0$ and $X$ is projective (see \cite[Thm. 3.17]{Binda2018}). Similar questions may be asked for $j\geq 1$.

\begin{lemma}\label{lemma_BK_LW_dim1_pic}\cite[Prop. 1.4]{LeWe85}
Suppose that $X$ is reduced and purely $1$-dimensional. Then there is a canonical
isomorphism $$\CH_0(X,Y) \cong \CH^{LW}_0(X,Y) \cong \Pic(X).$$
\end{lemma}

\begin{remark}
If a variety $X$ over a field $k$ is not regular, the Milnor K-sheaf $\cal K^M_{j,X}$ does not satisfy the Gersten conjecture. Nevertheless, one expects a Bloch-Quillen type formula to hold if one replaces the ordinary Chow group by the Levine-Weibel Chow group:
$\CH^{LW}_0(X,j)\cong H^d(X,\hat K^M_{j+d}).$ For $j=0$ this has been studied extensively, see for example \cite{GuptaKrishna2020}. Proposition \ref{Bloch_formula_dim_1} is an instance of this conjecture for $j\geq 0$ in dimension one.
\end{remark}

\subsection{Comparison with cdh-motivic cohomology}\label{subsection_cdh–motivic}

If $X$ is a smooth variety, then motivic cohomology coincides with Bloch's higher Chow groups:
$$\CH^{i}(X,2i-n)\cong H^{n,i}(X,\bb Z)(\cong H^{n,i}_{\rm cdh}(X,\bb Z)).$$
Taking the cohomology of the motivic complex $\bb Z(n)$ in the cdh-topology, motivic cohomology was extended to singular schemes by Suslin and Voevodsky in \cite[Sec. 5]{SV00}. In this section we relate these groups in the range of zero cycles and with finite coefficients to the groups defined in Section \ref{section_cohomological_theory_of_cycle_complexes}.

For the rest of this section let $X$ be a simple normal crossing variety and let $\epsilon:X_{\rm cdh}\to X_{\rm Nis}$ be the change of sites. Let $\Lambda(j)_{X,\rm Nis}:=R\epsilon_*\Lambda(j)_X$.

\begin{proposition}\cite[Prop. 8.1]{KerzEsnaultWittenberg2016}\label{proposition_comparison_sheaves_KEW}
For $\Lambda =\bb Z/n\bb Z$ with $n$ prime to $p$, there is a canonical
isomorphism of Nisnevich sheaves
$$S : \cal K_{X,j}^M/n\to \cal H^j(\Lambda(j)_{X,Nis}).$$
For $i>j$, the sheaf $\cal H^i(\Lambda(j)_{X,Nis})$ vanishes.
\end{proposition}

\begin{corollary}\label{corollary_App_B}
Let $d=\dim X$.
Then $$H^d_{\rm Nis}(X,\cal K^M_{j+d}/n)\cong H_{\rm cdh}^{2d+j,d+j}(X,\bb Z/n)$$
for all $j\geq 0$. \end{corollary}
\begin{proof}
We show that the first map in the composition
$$H^d_{\rm Nis}(X,\cal K^M_{j+d}/n)\to H_{}^{2d+j}(X,\Lambda(d+j)_{Nis})\cong H_{\rm cdh}^{2d+j,d+j}(X,\Lambda)$$ is an isomorphism.
Consider the spectral sequence
$$E_2^{r,s}(X)=H^r_{\rm Nis}(X,\cal H^s(\Lambda(j+d)_{\rm Nis})\Rightarrow H^{r+s}(X,\Lambda(j+d)_{\rm Nis}).$$
Then $H^{r}(X,\cal H^{>j+d}(\Lambda(j+d)_{\rm Nis}))=0$ by Proposition \ref{proposition_comparison_sheaves_KEW} and $H^{>d}(X,\cal H^s(\Lambda(j+d)_{\rm Nis})=0$ since the Nisnevich dimension is the dimension of the scheme.
Therefore $E_2^{d,d+j}(X)=H^d_{\rm Nis}(X,\cal H^{d+j}(\Lambda(j+d)_{\rm Nis})\cong H^{2d+j}(X,\Lambda(d+j)_{\rm Nis}). $
\end{proof}

\begin{remark}
An important property of cdh-motivic cohomology is the Mayer-Vietoris property for closed covers: let $X=\bigcup_i^r X_i$ and $X'=X_2\cup\dots\cup X_{r-1}$. Then the sequence
$$\dots \to H^{i-1,j}_{\rm cdh}(X_1)\oplus H^{i-1,j}_{\rm cdh}(X')\to H^{i-1,j}_{\rm cdh}(X_1\cap X') \to H^{i,j}_{\rm cdh}(X)$$ $$\to H^{i,j}_{\rm cdh}(X_1)\oplus H^{i,j}_{\rm cdh}(X')\to H^{i,j}_{\rm cdh}(X_1\cap X')\to \dots$$
is exact. We remark that this sequence does not directly relate to the cohomology groups of Section \ref{section_cohomological_theory_of_cycle_complexes}. We do not know if Corollary \ref{corollary_App_B} also holds in the Chow range, i.e. if there is an isomorphism
$H^r_{\rm Nis}(X,\cal K^M_{r}/n)\cong H_{\rm cdh}^{2r,r}(X,\bb Z/n)$
for all $r\geq 0$.
\end{remark}

\bibliographystyle{acm}
\bibliography{Bibliografie}

\noindent
\parbox{0.5\linewidth}{
\noindent
Morten L\"uders \\ 
Universität Heidelberg\\
Mathematisches Institut \\
Im Neuenheimer Feld 205 \\
69120 Heidelberg \\
Germany\\
{\tt mlueders@mathi.uni-heidelberg.de}
}

\end{document}